\documentclass[smallextended]{svjour3}     
\smartqed  
\usepackage{graphicx}
\usepackage{amssymb,epsfig,amsmath, multirow}

\newtheorem{assumption}{Assumption}

\newcommand{\RR}{{\mathbb R}}
\newcommand{\ZZ}{{\mathbb Z}}
\newcommand{\NN}{{\mathbb N}}
\newcommand{\D}{{\mathcal D}}
\newcommand{\C}{{\mathcal C}}

\newcommand{\sD}{{\cal D}}
\newcommand{\sB}{{\cal B}}

\newcommand{\sF}{{\cal F}}
\newcommand{\sP}{{\cal P}}

\newcommand{\sS}{{\cal S}}

\newcommand{\sQ}{{\cal Q}}
\newcommand{\sZ}{{\cal Z}}
\newcommand{\sV}{{\cal V}}
\newcommand{\sM}{{\cal M}}
\newcommand{\sX}{{\cal X}}

\newcommand{\eeq}{\end{equation}}

\newcommand{\ra}{\rightarrow}
\newcommand{\Ra}{\Rightarrow}

\newcommand{\deq}{\stackrel{\rm d}{=}}
\newcommand{\beql}[1]{\begin{equation}\label{#1}}
\newcommand{\eqn}[1]{(\ref{#1})}
\newcommand{\beq}{\begin{displaymath}}
\newcommand{\eeqno}{\end{displaymath}}

\newcommand{\lm}{\lambda}
\newcommand{\ep}{\epsilon}

\newcommand{\qandq}{\quad\mbox{and}\quad}

\newcommand{\qiffq}{\quad\mbox{if and only if}\quad}

\newcommand{\qasq}{\quad\mbox{as}\quad}

\newcommand{\qinq}{\quad\mbox{in}\quad}

\newcommand{\bes}{\begin{equation*}}
\newcommand{\ees}{\end{equation*}}
\newcommand{\bequ}{\begin{equation}}
\newcommand{\bi}{\begin{itemize}}
\newcommand{\ei}{\end{itemize}}
\newcommand{\bsplit}{\begin{split}}
\newcommand{\esplit}{\end{split}}
\newcommand{\bea}{\begin{eqnarray}}
\newcommand{\eea}{\end{eqnarray}}
\newcommand{\beas}{\begin{eqnarray*}}
\newcommand{\eeas}{\end{eqnarray*}}
\newcommand{\btab}{\begin{tabular}}
\newcommand{\etab}{\end{tabular}}

\newcommand{\barx}{\bar{X}}

\newcommand{\hatq}{\hat{Q}}
\newcommand{\hata}{\hat{A}}
\newcommand{\hatu}{\hat{U}}
\newcommand{\hatL}{\hat{L}}
\newcommand{\hats}{\hat{S}}
\newcommand{\hatd}{\hat{D}}
\newcommand{\hatv}{\hat{V}}
\newcommand{\hati}{\hat{I}}
\newcommand{\hatz}{\hat{Z}}

\def\lm{\lambda}
\def\tinf{\rightarrow\infty}

\def\AA{\mathbb{A}}
\def\SS{\mathbb{S}}

 \journalname{Queueing Systems}
\begin{document}

\title{{\small 05/03/12, revision 11/13/12} \\
Diffusion approximation for an overloaded X model via a stochastic averaging principle}

\titlerunning{FCLT via An Averaging Principle}        

\author{Ohad Perry \and Ward Whitt}

\authorrunning{Perry and Whitt} 

\institute{Ohad Perry \at
              Department of Industrial Engineering and Management Sciences, Northwestern University, Evanston, IL 60208, USA \\
              Tel.: +847-467-443\\
              Fax: +847-491-8005\\
              \email{ohad.perry@northwestern.edu}           
           \and
           Ward Whitt \at
              Department of Industrial Engineering and Operations Research, Columbia University, New York, NY 10027-6699, USA \\
                Tel.: +212-854-7255\\
              Fax: +212-854-8103\\
              \email{ww2040@columbia.edu}           
              }
\date{Received: date / Accepted: date}

\maketitle

\begin{abstract}
In previous papers we developed a deterministic fluid approximation
for an overloaded Markovian queueing system having two customer
classes and two service pools, known in the call-center
literature as the X model.
The system uses the
fixed-queue-ratio-with-thresholds (FQR-T) control, which we
proposed as a way for one service system to
help another in face of an unexpected overload.
Under FQR-T, customers are served by their own service pool until
a threshold is exceeded. Then, one-way sharing is activated
with customers from one class allowed to be served in both
pools. The control aims to keep the two queues at a
pre-specified fixed ratio.
We supported the fluid approximation
by establishing a functional weak law of large numbers (FWLLN) involving a stochastic averaging principle.
In this paper we develop a refined diffusion approximation
for the same model based on a many-server heavy-traffic functional central limit theorem (FCLT).
\end{abstract}

\section{Introduction}\label{secIntro}

In this paper we establish a many-server heavy-traffic {\em functional central limit theorem} (FCLT)
for an overloaded large-scale Markovian queueing system having two classes
and two service pools, known as the $X$ model \cite{GKM03}, using the {\em fixed-queue-ratio with thresholds} (FQR-T) routing,
which we proposed in \cite{PeW09a}.

In particular, we consider a system in which each class has its own designated service pool, but with all agents, in both pools,
capable of serving customers from both classes.
The control aims to prevent sharing of customers (i.e., sending customers from one class to be served at the other class pool)
when both classes are normally loaded, and to activate sharing when the system unexpectedly experiences an overloaded, due to an unforseen shift in
the arrival rates. 

When sharing is taking place, the control aims at keeping a
pre-specified fixed ratio between the two queues, both as the overload develops over time and in the overload steady state. This ratio is
chosen according to an optimization
problem for the approximating stationary deterministic ``fluid'' model, assuming
a convex holding cost is incurred on the two queues during the overload
incident; see \S 5.3 in \cite{PeW09a}, where it is also shown
that sharing should not be allowed at both directions
simultaneously, i.e., at any time there should be at most one
pool working with both classes.
In general, there are two different ratios:
If class $1$ is overloaded, then an optimal ratio $r_{1,2}$ should hold
between the queues. If class $2$ is overloaded, then an optimal ratio $r_{2,1}$ should hold between the queues.
In \cite{PeW10b} we showed that the FQR-T control achieves the target ratios asymptotically as the scale increases (in the fluid limit),
for the time-dependent transient performance as well as in steady state.
Moreover, the FQR-T control produces a tractable fluid limit.
Here we show that the FQR-T control also produces a tractable refined stochastic limit.



The FQR-T control here is a modification of the FQR control (without the thresholds), which is a special case of
the {\em queue-and-idleness ratio} (QIR) controls suggested by Gurvich and Whitt \cite{GW09b}.
These QIR and FQR controls were analyzed in \cite{GW09a}, \cite{GW09b} and \cite{GW10} for critically loaded systems, operating in the
{\em quality and efficiency driven} (QED) many-server heavy-traffic regime; see \cite{GMR02,HW81}.
Heavy-traffic limits for networks having cyclic graphs, such as the X model, were obtained under the condition that the service rates
are class or pool dependent; see Theorem 3.1 in \cite{GW09b}.
In general, when the service rate depends on both the class and the pool, FQR can perform badly in cyclic networks, creating severe congestion
even if each pool is not congested by itself; see \S 4.1 in \cite{PeW09a} and \S EC.2 in \cite{PeW09b}.

We suggested the FQR-T control in \cite{PeW09a}, and analyzed the X model using a stationary fluid approximation.
In \cite{PeW09b} we determined the transient behavior of that same fluid model,
based on a stochastic {\em averaging principle} (AP), but that AP was introduced there as a heuristic engineering principle (i.e., used without proof),
supported only by simulation. This heuristic analysis and the AP were made rigorous in our subsequent papers. In particular,
the purpose of \cite{PeW10a,PeW10b} was to establish key mathematical properties of the fluid model,
expressed as an {\em ordinary differential equation} (ODE),
and show that the fluid model in \cite{PeW09a,PeW09b},
arises as the many-server heavy traffic limit of a sequence of $X$ models in the many-server {\em efficiency driven} (ED) regime.
That FWLLN is challenging, because the fluid limit depends critically on the AP.
For each $n$, the system evolves as a $6$-dimensional {\em continuous-time Markov chain} (CTMC), but
there is (a somewhat complicated) statistical regularity associated with the many-server heavy-traffic limit.
In particular, the limiting fluid approximation is a deterministic function characterized by an ODE (and an initial condition),
which is driven by the time-varying instantaneous average behavior of a family of {\em fast-time-scale stochastic processes} (FTSP's),
which produces the AP.
See \S 1.3 of \cite{PeW10b} for a discussion of the literature on AP's;
notable contributions in the queueing literature are by Coffman et al. \cite{CPR95} and Hunt and Kurtz \cite{HK94}.
See \cite{FRT03} for a (quite different) FCLT involving an AP, building on \cite{HK94}.



We now build on the FWLLN and the AP to describe the
distribution of the stochastic fluctuations about the fluid
path; i.e., we establish the corresponding FCLT, which is
Theorem \ref{thDiffLimit} here.  There is technical novelty in
properly treating the FTSP's alluded to above. The limit
process involves an independent Brownian motion term with
deterministic time scaling involving the asymptotic variance of
the FTSP; see \S \ref{secFTSPinFCLT} and $\hatL_2$, $\hati$,
$\gamma_2$ and $\gamma_3$ in Theorem \ref{thDiffLimit}.
A key step in establishing the main result -- the FCLT in Theorem \ref{thDiffLimit} -- is a FCLT for the family of FTSP's, Theorem \ref{lmTvFCLT},
which is of independent interest.
This challenging step proves a FCLT for a sequence of CTMC's having time-varying parameters depending on the fluid limit.
The new methods developed here should prove useful for analyzing related problems.

From an engineering perspective, Corollary \ref{corDiffLim} is especially useful for understanding the performance of the FQR-T control.
It describes the stochastic-process limit once the fluid has stabilized (i.e. when the fluid is stationary).
With a constant fluid state, the key limit process becomes the well-studied {\em bivariate Ornstein-Uhlenbeck} (BOU) process, which has a Gaussian distribution for each $t$;
 see Corollary \ref{corDiffLim} below.
Consequently, the approximating steady-state distribution during the overload is a Gaussian distribution, with mean values equal to the
stationary fluid point in Theorem \ref{thFluidStat} multiplied by $n$, and variance and covariance terms
in \eqn{covMatrix} multiplied by $\sqrt{n}$.

The FCLT extension is essential for truly understanding the system performance under overloads,
because the actual performance is not nearly deterministic, as described by the fluid approximation,
unless the scale is extremely large.  This phenomenon is well illustrated by the example here in \S \ref{secSim}.
For that example, the standard deviations of the queue lengths are about equal to (half of) the mean queue lengths
when the number of servers in each pool is $25$ ($100$).

Here is how the paper is organized:  After preliminaries in \S \ref{secRep}, we
briefly state the FWLLN and the associated WLLN for the stationary distributions in \S \ref{secFluid}.
We state the FCLT and our other main results in \S \ref{secFCLT}.
We prove the FCLT in \S \ref{secDiffProof} except for
 Lemma \ref{lmLim}, establishing joint convergence of the driving processes.
We give the proof of Lemma \ref{lmLim} in \S \ref{secProofLim} except for two supporting results.
The key supporting result is a FCLT for the FTSP with time-varying parameter state function in Theorem \ref{lmTvFCLT}.
We prove Theorem \ref{lmTvFCLT} in \S \ref{secProofQBDFCLT}.
Our proof of Lemma \ref{lmRandSumBd} to prove Theorem \ref{lmTvFCLT} exploits the martingale FCLT for triangular arrays.
We state these supporting martingale results in \S \ref{secMgFCLT}.
We then prove five remaining lemmas in \S \ref{secProofsLemmas}.
A key technical step in the proofs is approximating the given process with time-varying parameters over appropriate subintervals by associated
{\em frozen} processes, where the parameters are fixed (frozen) at designated values.
Those approximation steps are justified in \S \ref{secFrozenPf} by using coupling constructions.
In particular, we prove Lemmas \ref{lmNewFrozen} and \ref{lmCumFzn} there.
Finally, we evaluate the quality of the approximations by making comparisons with simulations in \S \ref{secSim}.

\section{Preliminaries} \label{secRep}

\subsection{Notation}

Let $\RR$, $\ZZ$ and $\NN$ denote the real numbers, integers and nonnegative integers, respectively.
Let $\equiv$ denote equality by definition.  For a subinterval $I$ of $[0,\infty)$,
let $\D \equiv \D(I) \equiv \D(I, \RR)$
be the space of all right-continuous $\RR$-valued functions on $I$ with limits from the left everywhere,
endowed with the familiar Skorohod $J_1$ topology \cite{W02}.  Let $\C$ be the subset of continuous functions in $\D$.
Let a subscript $k$ appended to one of these spaces denote the set of all $k$-dimensional vectors with components from the space,
endowed with the corresponding product topology, e.g., $\RR_k$ and $\D_k$.

Let $d_{J_1}$ denote a metric on $\D_k(I)$ inducing the convergence.
Since we will be considering continuous limits, the topology is equivalent to uniform convergence on compact subintervals of $I$.
  Let $e$ be the identity function in $\D \equiv \D_1$; i.e.,
$e (t) \equiv t$, $t \in I$.  Let $\circ$ be the composition function, i.e., $(x\circ y) (t) \equiv x(y(t))$.
Let $\Rightarrow$ denote convergence in distribution \cite{W02}.

We use the familiar big-$O$ and small-$o$ notation for deterministic functions:
For two real functions $f$ and $g$, we write
\bes
\bsplit
f(x) & = O(g(x)) \quad \mbox{whenever} \quad 0 < \limsup_{x \ra \infty} |f(x) / g(x)| < \infty, \\
f(x) & = o(g(x)) \quad \mbox{whenever} \quad \limsup_{x\tinf} |f(x) / g(x)| = 0.
\end{split}
\ees
(Note that our definition of $O(g(x))$ deviates from the standard definition which allows for the $\limsup$ in the right-hand side to be equal to $0$.)
For a function $x : [0, \infty) \ra \RR$ and $0 < t < \infty$, let
$\|x\|_t \equiv \sup_{0 \le s \le t} | x(s)|.$

For a stochastic process $Y \equiv \{Y(t) : t \ge 0\}$ and a deterministic function $f : [0, \infty) \ra [0, \infty)$,
we say that $Y$ is $o_P(f(t))$ if $\|Y\|_t/f(t) \Rightarrow 0 \qasq t \tinf$.

For a sequence of stochastic processes or random variables, $\{Y^n : n \ge 1\}$,
we denote its fluid-scaled version by $\bar{Y}^n \equiv Y^n / n$.
We let $\breve{Y}^n \equiv Y^n / \sqrt{n}$ be the $\sqrt{n}$-scaled processes without the centering about the fluid limit,
and $\hat{Y}^n$ denote the diffusion-scaled processes centered about the fluid limit, as in \eqn{DiffScale} below.

\subsection{A Sequence of Overloaded Markovian X Models}\label{secModel}

We consider a sequence of overloaded Markovian X models, indexed by superscript $n$.
There are two customer classes and two service pools.
We are looking at these models during the overload incident, after the arrival rates have changed.
The arrival rates are considered fixed, but the system
is typically not yet in its new steady-state during the overload (assuming that the overload would persist).
For each $n$ and $i = 1,2$, there is a class-$i$ Poisson arrival process with rate $\lambda^n_i$.
Customers have limited patience, and may abandon when waiting in queue.
The times to abandon are i.i.d. exponential variables with rate $\theta_i$ for each class-$i$ customer in queue.
Service pool $j$ has $m^n_j$ homogeneous agents (servers).
Service times of class-$i$ customers by pool-$j$ agents are mutually independent and exponentially distributed with rate $\mu_{i,j}$, $i,j = 1,2$.
The abandonment and service rates are independent of $n$.
We mention that we make no assumptions on the four service rates and, in particular, we do not assume that the {\em weak inefficiency} condition holds,
namely, that $\mu_{1,1}\mu_{2,2} \ge \mu_{1,2} \mu_{2,2}$. This condition was key to our analysis in \cite{PeW09a}, but here we study a more general system.

Since we are considering an overload incident, we will scale to achieve an efficiency-driven (ED)
many-server heavy-traffic regime.
\begin{assumption}{$($many-server heavy-traffic scaling$)$} \label{AssDiff}
For $\lm_i, m_i > 0$, $i = 1,2$,
\bes
\frac{\lm^n_i - n\lm_i}{\sqrt{n}} \ra 0 \qandq \frac{m^n_i - n m_i}{\sqrt{n}} \ra 0 \qasq n \tinf.
\ees
\end{assumption}
We could instead obtain a modified, more general, FCLT if there were nondegenerate limits in Assumption \ref{AssDiff},
but we consider our choice natural, because the system operates in an overload regime.
(The modified limit includes a deterministic term $ct$ in the diffusion limit, but there is no difference in the variability of the limit process,
as can be seen from \eqn{center}.
For the FWLLN, it is sufficient that $\lm^n_i/n \ra \lm_i$ and $m^n_i/n \ra m_i$ as $n \tinf$, $i = 1,2$.)

Let
\bes
\rho_i \equiv \frac{\lm_i}{m_i \mu_{i,i}} \qandq q_i^a \equiv \frac{(\lm_i - \mu_{i,i}m_i)^+}{\theta_i}, \quad i = 1,2,
\ees
where, for $y \in \RR$, $y^+ \equiv \max\{0, y\}$.
Then $\rho_i$ is the traffic intensity for pool $i$ and $q_i^a$ is the stationary class-$i$ fluid-limit queue,
when both pools operate independently.
We say that pool $i$ is overloaded if $\rho_i > 1$. However, with sharing allowed, pool $i$ can be overloaded even if $\rho_i < 1$ provided
that enough class $j$ customers are routed to be served there, $j \ne i$.
The next assumption makes precise our notion of system overload.

\begin{assumption}{$($system overload, with class $1$ more overloaded$)$} \label{AssOverload} \\
The rates in the system are such that \\
$(I)$ $\theta_1 q_1^a > \mu_{1,2}m_2 (1-\rho_2)^+$ \quad and \quad $(II)$ $q_1^a > r_{1,2} q_2^a$.
\end{assumption}
Clearly, $\rho_1 > 1$ by Condition $(I)$, so that class $1$ is overloaded.
However, Condition $(I)$ also ensures that pool $2$ is overloaded
if sharing is taking place. That is so because, even if $\rho_2 < 1$, there is not enough extra service capacity in pool $2$ to take care
of all the class-$1$ customers that pool $1$ cannot serve.
Condition $(II)$ in the assumption implies that even if pool $2$ is overloaded by itself (i.e., if $\rho_2 > 1$),
then class $1$ is the one that should receive help from pool $2$.


\subsection{The FQR-T Control} \label{secFQRT}

We now describe the FQR-T control for each system $n$.
The purpose of the FQR-T control is:  (i) to prevent sharing under normal loads, (ii) to activate sharing as soon as an overload incident
begins, and (iii) to keep close to the desired ratio between the two queues, making sure that sharing takes place in the needed direction only.
The control is based on two positive thresholds, $k^n_{1,2}$ and $k^n_{2,1}$,
and the two ratio parameters discussed above, $r_{1,2}$ and $r_{2,1}$, which satisfy $r_{1,2} \ge r_{2,1}$;
see Proposition EC.2 and Equation (EC.11) in \cite{PeW09a}.


Let $Q^n_i (t)$ be the number of customers in the class-$i$ queue and let
$Z^n_{i,j} (t)$ be the number of class-$i$ customers being served in service pool $j$,
at time $t$, $i,j = 1,2$ (in the $n^{\rm th}$ system).  The FQR-T routing is based on the queue-difference stochastic processes
\bequ \label{QD}
\bsplit
D^n_{1,2} (t) & \equiv Q^n_1 (t) - k^n_{1,2} - r_{1,2} Q^n_2 (t), \qandq \\
D^n_{2,1} (t) & \equiv r_{2,1} Q^n_2(t) - k^n_{2,1} - Q^n_1(t), \quad t \ge 0.
\end{split}
\eeq
As long as $D^n_{1,2} (t) \le 0$ and $D^n_{2,1} (t) \le 0$,
 no sharing of customers is allowed, i.e.,
a server in pool $j$ takes only class $j$ customers, $j = 1,2$.
It follows from \cite{GMR02} that thresholds of order larger than $O(\sqrt{n})$ will prevent sharing (asymptotically, as $n \ra \infty$) when both pools are normally loaded,
because normally loaded systems, that are not overloaded, have stochastic fluctuations that are of order $O(\sqrt{n})$.
Once one of the queue-difference processes in \eqref{QD} becomes strictly positive (so that one of the thresholds is crossed)
sharing is initiated.  It follows from the Corollary 2.1 in \cite{W04}, that thresholds of size $o(n)$ will detect an overload
relatively quickly (instantly, asymptotically as $n \tinf$). This is because overloaded queues are of order $n$ asymptotically.
We thus choose the thresholds according to the following assumption.

\begin{assumption}{$($scaling of the thresholds$)$} \label{AssThresholds}
For $k_{1,2}, k_{2,1} > 0$ and a sequence of positive numbers $\{c_n : n \ge 1\}$, where $c_n/n \ra 0$ and $c_n/\sqrt{n} \ra \infty$ as $n \tinf$,
\bes
k^n_{1,2}/c_n \ra k_{1,2} \mbox{ and } k^n_{2,1}/c_n \ra k_{2,1} \qasq n \ra \infty.
\ees
\end{assumption}

Finally, only one-way sharing is allowed at any time.
For example, a newly available pool-$2$ agent at time $t$ serves a class-$1$ customer if $D^n_{1,2} (t) > 0$, provided no class-$2$ customers
are served in pool $1$ at that same time $t$; otherwise he serves a class-$2$ customer.

\subsection{Dimension Reduction} \label{secDimRed}

For the X model operating under FQR-T, the six-dimensional process
\bequ \label{Xn6}
X^n_6 \equiv (Q^n_1, Q^n_2, Z^n_{1,1}, Z^n_{1,2}, Z^n_{2,1}, Z^n_{2,2})
\eeq
is a CTMC for each $n \ge 1$.
However, there is an important dimension reduction established in \S 6 of \cite{PeW10b}.
It was shown, under the assumptions above and with appropriate initial conditions, that asymptotically the two service pools
remain fully occupied with no pool-$1$ servers serving class $2$; i.e.,
for each $T > 0$,
\beq
P(Z^n_{1,1} (t) = m^n_{1}, Z^n_{2,1} (t) = 0, Z^n_{1,2} + Z^n_{2,2} = m^n_2, 0 \le t \le T) \ra 1 \qasq n \ra \infty.
\eeqno
Thus, the system is characterized by an essentially three-dimensional process
\bequ \label{Xn*}
X^{n,*}_6 \equiv (Q^n_1, Q^n_2, m^n_1, Z^n_{1,2}, 0, m^n_2 - Z^n_{1,2}), 
\eeq
having the vector of essential components
\bequ \label{Xn}
X^n \equiv (Q^n_1, Q^n_2, Z^n_{1,2}),
\eeq
whose evolution is directly specified, and will be specified here in Theorem \ref{thSSC1}.
Theorem \ref{thSSC1} concludes that
$\barx^{n,*}_6$ and $\barx^n_6$ are asymptotically equivalent,
so that $\barx^n$ is sufficient to characterize the FWLLN and, in turn, to prove the FCLT.
That implies that $\barx^n_6 \Ra x_6$ in $\D_6$ if and only if $\barx^n \Ra x$ in $\D_3$ as $n \tinf$,
with $x(t) \in \SS \equiv [0, \infty)^2 \times [0, m_2]$, for $t \ge 0$; see Theorem \ref{th1} below.
We thus restrict attention to the space $\SS$.


\subsection{The Fast-Time-Scale Process} \label{secFTSP}

Given that the system is overloaded with class $1$ needing help from pool $2$, as determined
by Assumptions \ref{AssDiff} and \ref{AssOverload}, the FQR-T control is driven by the process $D^n_{1,2}$ in \eqref{QD}.
Since the queue lengths are asymptotically of order $O(n)$, the queue-difference process $D^n_{1,2}$ has transitions at rate $O(n)$.
However, Theorem 4.5 in \cite{PeW10b} shows that, under regularity conditions, the sequence $\{D^n_{1,2} (t): n \ge 1\}$ is stochastically bounded in $\RR$,
so that the difference process should be analyzed without any spatial scaling. On the other hand, Theorem 4.4 in \cite{PeW09b} also shows that
this sequence is {\em not} $\D$-tight.
Thus, these difference processes do not converge to nondegenerate limits in $\D$ as $n \ra \infty$ without spatial scaling.
Nevertheless, both the FWLLN and FCLT depend heavily on the asymptotic behavior of functionals of that driving queue-difference process
and on the analysis of a related family of {\em fast time scale process} (FTSP's).

Fix $t_0 \ge 0$ and consider the {\em time expanded queue-difference process}
\bequ \label{expanded}
\{D^n_e (\Gamma^n, s): s \ge 0\} \equiv \{D^n_{1,2}(t_0 + s/n) : s \ge 0\},
\eeq
where $\Gamma^n$ is a random vector in $\RR_3$, representing a possible state of $X^n$, and we condition on $X^n(t_0) = \Gamma^n$.
Theorem 4.4 in \cite{PeW10b} shows, under the assumptions of the FWLLN in Theorem \ref{th1} below,
that
\bequ \label{th53}
\{D^n_e (\Gamma^n, s): s \ge 0\} \Ra \{D(\gamma, s): s \ge 0\} \qinq \D \qasq n\tinf
\eeq
if $\Gamma^n/n \Ra \gamma \in \SS$ and $D^n_e (\Gamma^n, 0) \Ra D(\gamma, 0)$ in $\RR$ as $n \tinf$.  The limit process $D(\gamma, \cdot)$ is the FTSP,
an irreducible pure-jump (time homogeneous) Markov process having transition rates that are the limit of the instantaneous rates of
$D^n_{1,2} (t_0)$ at time $t_0$ (given the state of the CTMC $X^n_6 (t_0)$), divided by $n$.
Since the distribution of the FTSP is determined
by $\gamma$, we obtain a different FTSP $D(\gamma, \cdot)$ for each $\gamma \in \SS$, and thus for each $t \ge 0$.
The name ``FTSP'' becomes clear when observing that it arises as the limit in \eqref{expanded}
achieved by ``slowing'' time in the neighborhood of each time point
$t_0$ in $D^n_{1,2}(t_0)$.

As explained in \S \ref{secFQRT},
the purpose of the FQR-T control during overload periods (with class $1$ receiving help) is to keep the two queues approximately fixed at the target ratio $r$.
In this paper we will be concerned with the region of the state space in which $q_1 = r q_2$ and the FTSP is positive recurrent.
In particular, for $\gamma \equiv (q_1, q_2, z_{1,2})$ we let
\bes
\SS^b \equiv \{\gamma \in \SS : q_1 = r q_2\}
\ees
denote the `boundary' set of points in $\SS$ which is part of the state space to which the control drives the process.
We then let $\AA$ denote the set of all $\gamma \in \SS^b$, such that $D(\gamma, \cdot)$ is positive recurrent, with
$D(\gamma, \infty)$ denoting a random variable distributed as the stationary distribution of the FTSP $D(\gamma, \cdot)$.
For each $\gamma \in \SS^b$, let
\bequ \label{pi}
\pi_{1,2}(\gamma) \equiv P (D(\gamma, \infty) > 0).
\eeq
By Lemma 3.1 in \cite{PeW10b}, $\pi_{1,2}(\gamma)$ is well defined for all $\gamma \in \SS$, but
$D(\gamma, \cdot)$ is positive recurrent if and only if $0 < \pi_{1,2}(\gamma) < 1$ {\bf and} $\gamma \in \SS^b$.
By Theorem 6.1 of \cite{PeW10a},
\bequ \label{posrec}
\AA = \{\gamma \in \SS^b: 0 < \pi_{1,2}(\gamma) < 1\} = \{\gamma \in \SS^b: \delta_{+} (\gamma) < 0 \qandq \delta_{-} (\gamma) > 0\},
\eeq
where $\delta_{+} (\gamma)$ and $\delta_{-} (\gamma)$, respectively, are the constant drift rates in the positive region
$\{s:D(\gamma, s) > 0\}$ and the non-positive region $\{s:D(\gamma, s) \le 0\}$.


Both the FWLLN and the FCLT depend critically on distributional and topological characteristics of the FTSP's.
A simplification is achieved by representing the FTSP as a {\em quasi-birth-and-death} (QBD) process,
which can be done by assuming that $r_{1,2}$ is rational. The QBD representation is not straightforward, thus we
refer to \S 6.2 in \cite{PeW10a} for more details on the QBD representation of the FTSP, and to \cite{LR99} for the general theory of QBD processes.
See also Theorem 6.1 and Equation (7.2) in \cite{PeW10a} for how the QBD representation
simplifies the characterization of $\AA$, as well as \S 11 in \cite{PeW10a}, where an efficient algorithm for computing the fluid
limit numerically is developed, based on that QBD representation.
For our purposes here, it only matters that the FTSP can be analyzed as a QBD, provided that the queue ratios
are rational number.
We thus make the following assumption.

\begin{assumption}{$($queue ratios parameters$)$} \label{AssRatio}
The queue ratios $r_{1,2}$ and $r_{2,1}$ are positive rational numbers.
\end{assumption}

Since we are considering the case when sharing is taking place with class-$1$ customers receiving help, we essentially need only consider
$r_{1,2}$, which we henceforth denote by $r$, i.e., $r \equiv r_{1,2}$.

\section{The Fluid Limit} \label{secFluid}

We now review the FWLLN for the process $\barx^n_6$ in \eqref{Xn6} and the WLLN for the associated sequence of
stationary random variables $\barx^n_6 (\infty)$, established in \cite{PeW10b}.
For these, we assume that the fluid $x(t)$ is in the set $\AA$, where the FTSP is positive recurrent.
We conclude by reviewing a result stating that the fluid model eventually remains in $\AA$.

\subsection{The FWLLN}

We now describe the fluid limit, i.e., the limit of $\barx^n_6$ for $X^n_6$ in \eqref{Xn6}.
The FWLLN requires an assumption about the initial conditions.  In \cite{PeW10b} we considered a (more general) version of the following.
\begin{assumption}\label{AssInitial}
 Assume that
\bes 
\bsplit
& P(Z^n_{2,1}(0) = 0, Q^n_i(0) > a_n, i = 1,2) = 1 \quad \mbox{for all $n \ge 1$}, \\
& \barx^n(0) \Ra x(0) \in \AA \qandq D^n_{1,2} (0) \Ra L \quad \mbox{as } n\tinf,
\end{split}
\ees
where $L$ is a finite random variable, $x(0)$ is deterministic and $\{a_n : n \ge 1\}$ is a sequence of numbers satisfying
$a_n/c_n \ra a$, $0 < a \le \infty$, for $c_n$ in Assumption {\em \ref{AssThresholds}}.
\end{assumption}
We note that in \cite{PeW10b} $x(0)$ was not necessarily in $\AA$.
The following theorem is a version of the main result - Theorem 4.1 - in \cite{PeW10b}, adapted to our needs here.
\begin{theorem} \label{th1}{$(${\em FWLLN}$)$}
Under Assumptions \ref{AssDiff}-\ref{AssInitial},
\beq
 \bar{X}^n_6 \Rightarrow x_6 \qinq \D_{6}([0, \infty)) \qasq n \ra \infty,
 \eeqno
  for $X^n_6$ in \eqref{Xn6},
where $x_6 \equiv (q_i, z_{i,j}; i,j = 1,2)$, is a deterministic element of $\C_{6}$, with $z_{1,1} = m_1e$, $z_{2,1} = 0e$ and
$z_{2,2} = m_2e - z_{1,2}$ and $x \equiv (q_1, q_2, z_{1,2})$ being the unique solution to the three-dimensional ODE
\bequ \label{odeDetails}
\bsplit
\dot{q}_{1} (t)   & \equiv  \lambda_1  - m_1 \mu_{1,1} - \pi_{1,2} (x(t))\left[z_{1,2} (t) \mu_{1,2} + z_{2,2} (t) \mu_{2,2}\right]
- \theta_1 q_1 (t) \\
\dot{q}_{2} (t)   & \equiv  \lambda_2   - (1 - \pi_{1,2}(x(t))) \left[z_{2,2} (t) \mu_{2,2} + z_{1,2} (t) \mu_{1,2}\right]
- \theta_2 q_2 (t) \\
\dot{z}_{1,2} (t) & \equiv  \pi_{1,2}(x(t)) z_{2,2} (t) \mu_{2,2} - (1 - \pi_{1,2}(x(t))) z_{1,2} (t) \mu_{1,2},
\end{split}
\eeq
for $\pi_{1,2}(x(t)) \equiv P(D(x(t), \infty) > 0)$ in \eqref{pi}.
Moreover, there exists $\delta$, $0 < \delta \le \infty$, such that $x(t) \in \AA$, so that
$0 < \pi_{1,2} (x(t)) < 1$ and $q_1 (t) = r q_2 (t)$, for all $t \in [0, \delta)$.
\end{theorem}
Just as the routing of customers at each time $t \ge 0$ in the prelimit
is determined by whether $D^n_{1,2}(t) > 0$ or $\le 0$, so also
the instantaneous future evolution of the fluid limit $x(t)$
at time $t \ge 0$, is determined by whether the FTSP corresponding to $x(t)$, $D(x(t), \cdot)$, is positive or nonpositive.
However, that evolution is determined by the {\em long-run average behavior}
of the FTSP corresponding to time $t$, i.e., by $\pi_{1,2}(x(t))$, giving rise to the term ``averaging principle''.
Loosely speaking, $D^n_{1,2}(t)$ achieves a local steady state (the steady state of the FTSP)
instantaneously as $n \tinf$, at each time $t \ge 0$.

Observe that Theorem \ref{th1} concludes that if $x(0) \in \AA$, then $x(t) \in \AA$ for all $t$ over some interval $[0, \delta)$ (that part of the theorem follows from
Theorem 4.5 in \cite{PeW10b}), so that we have SSC in the sense that the
original six-dimensional process is a deterministic function of a two-dimensional process.
More importantly for the FCLT, we also have that $Q^n_1(t) - k^n_{1,2} - r Q^n_2(t) = o(\sqrt{n})$
for $t \in (t_1, t_2)$ if $x(t) \in \AA$ over $[t_1, t_2)$, so the SSC to two dimensions holds in diffusion scale as well;
see Lemma \ref{thSSC2} below.

\subsection{The Stationary Fluid Limit} \label{secStatFluid}

Our main theorem here will be establishing the FCLT about the fluid trajectory, given that the trajectory is in $\AA$.
An important consequence will be the BOU limit when the fluid limit is stationary.
Since the fluid limit of $\barx^n$ in \eqref{Xn} is the unique solution to the ODE \eqref{odeDetails},
there is an immediate equivalence between stationarity of the fluid limit and stationarity of the dynamical system in \eqref{odeDetails},
and we do not distinguish between the two.

\begin{definition}{$($fluid stationarity$)$} \label{defStat}
A point $x^* \in \SS$ is a stationary point of the unique solution $x \equiv \{x(t) : t \ge 0\}$ to the ODE \eqref{odeDetails}
if $x(0) = x^*$ implies $x = x^* e$.
If $x = x^* e$, then $x$ is said to be stationary.
\end{definition}
Since the ODE is autonomous (i.e., time invariant), we can replace time $0$ with any $t > 0$ in the definition \ref{defStat}.
That is, if $x(T) = x^*$ for some $T > 0$, then $x(t) = x^*$ for all $t > T$.
Time invariance also implies that $x(t)$ is stationary at time $t$ ($x(t) = x^*$) if and only if
$\dot{x} (t) \equiv (\dot{q}_1 (t), \dot{q}_2 (t), \dot{z}_{1,2}(t)) = (0, 0, 0)$; see \S 8 of \cite{PeW10a}.

There are several issues regarding stationarity, which we addressed in \cite{PeW10a}.
In advance, neither existence of a stationary point to the fluid limit nor uniqueness are immediate.
Even if there exists a unique stationary point, it needs to be identified.
Moreover, it must be shown that the fluid limit converges to a stationary point
as $t \tinf$. (There are still other issues regarding stability of the dynamical system in \eqref{odeDetails},
and we refer to \S 8.3 in \cite{PeW10a} for a discussion.)
Finally, the fluid limit of $\barx^n_6$ in \eqref{Xn6} is characterized by the fluid limit of the three-dimensional $\barx^n$ in \eqref{Xn},
but that does not directly imply any relation between the stationary fluid limit and the stationary stochastic prelimit.

We now present the most relevant results for the FCLT regarding fluid stationarity.
\begin{theorem}{$($fluid stationarity$)$} \label{thFluidStat}
Under Assumptions \ref{AssDiff}-\ref{AssInitial}, the following hold:

$(i)$  For each $n$, $\barx^n_6 (t) \Ra \barx^n_6 (\infty)$ in $\RR$ as $t \ra \infty$, with $\barx^n_6 (\infty)$ being the
unique stationary distribution of the CTMC, and $\barx^n_6 (\infty) \Ra x^*_6$ in $\RR$ as $n \tinf$ for
\bequ \label{statPt}
x^*_6 \equiv (q^*_1, q^*_2, m_1, z^*_{1,2}, 0, m_2 - z^*_{1,2}),
\eeq
where
\bes
\bsplit
z_{1,2}^* & = \frac{\theta_2(\lm_1 - m_1\mu_{1,1}) - r \theta_1(\lm_2 - m_2\mu_{2,2})} {r \theta_1\mu_{2,2} + \theta_2\mu_{1,2}} \wedge m_2, \\
q_1^* & = \frac{\lm_1 - m_1\mu_{1,1} - \mu_{1,2} z_{1,2}^*}{\theta_1} \qandq
q_2^* = \frac{\lm_2 - \mu_{2,2} (m_2 - z_{1,2}^*)}{\theta_2}.
\end{split}
\ees

$(ii)$  $x^* \equiv (q^*_1, q^*_2, z^*_{1,2})$  is the unique stationary point of $x$, the unique solution to the ODE \eqref{odeDetails}.

$(iii)$ $\pi_{1,2}(x^*) \equiv P(D(x^*, \infty) > 0) = \pi^*_{1,2}$, where $D(x^*, \infty)$ is a random variable with the stationary distribution of the FTSP $D(x^*, \cdot)$ and
\bequ \label{piSS}
\pi^*_{1,2} \equiv \frac{\mu_{1,2}z^*_{1,2}}{\mu_{1,2}z^*_{1,2} + (m_2 - z^*_{1,2}) \mu_{2,2}}.
\eeq

$(iv)$  $x(t) \ra x^*$ as $t \tinf$ exponentially fast.
\end{theorem}

\begin{proof}
Parts $(i)$, $(ii)$  and $(iii)$, and $(iv)$, respectively, are covered by Theorem 4.2 in \cite{PeW10b}, \S 8 of \cite{PeW10a} and
Theorem 9.2 in \cite{PeW10a}.  Explicit exponential bounds on the rate of
convergence to stationarity in $(iv)$ are given in \cite{PeW10a}.  We now elaborate on $(ii)$ and $(iii)$.
First, if $x^* \notin \AA$, then the fact that $x^*$ is a stationary point of $x$ follows immediately
from the fact that $\pi_{1,2}(x^*) = 0$ or $= 1$. In that case, it is also easy to see that $\pi^*_{1,2}$ in \eqref{piSS}
is equal to $\pi_{1,2}(x^*)$; see Corollary 8.1 in \cite{PeW10a}.
It is the unique stationary point by Theorem 8.1 in \cite{PeW10a}.
The more challenging case, in which $x^* \in \AA$ and the existence of a stationary point is nontrivial,
is proved in Theorem 8.2 in \cite{PeW10a}.
\hfill \qed \end{proof}

\subsection{Eventually Remaining in the Set where the FTSP is Positive Recurrent}\label{secRemain}

The FCLT will be stated under the assumption that the associated fluid limit lies in the set $\AA$.
Thus we now explain why this makes sense and introduce an additional assumption.

Note that $x^*_6$ in \eqn{statPt} is completely characterized by $x^*$, which
involves only the rates in the system, and does not require any knowledge of the transient fluid limit
or the initial condition. (In particular, SSC to three dimensions holds for the WLLN of the stationary distributions.)
Simple algebra shows that if $0 < z^*_{1,2} < m_2$, then $q^*_1 = r q^*_2$. Together with \eqref{posrec} and \eqref{piSS}
we see that $x^* \in \AA$ if and only if $0 < z^*_{1,2} < m_2$.
It follows from Assumption \ref{AssOverload} and \eqref{statPt} that $z^*_{1,2} > 0$ (see also Corollary 8.2 in \cite{PeW10a}),
so that, under Assumption \ref{AssOverload},
\bequ \label{x*inA}
x^* \in \AA \qiffq z^*_{1,2} < m_2.
\eeq
The next theorem, which follows from Theorem 10.2 in \cite{PeW10a}, shows that there is not much loss in assuming that the limit $x$ lies entirely in $\AA$ whenever $x^* \in \AA$.

\begin{theorem} \label{thInA}
If $x^* \in \AA$ then there exists $T_A < \infty$ such that $x (t) \in \AA$ for all $t \ge T_A$.
\end{theorem}

Since we are interested in the case $x^* \in \AA$, which is the main case, as is clear from \eqref{x*inA},
we make the following assumption
\begin{assumption} \label{AssInA}
For all $t \ge 0$, $x(t) \in \AA$.
\end{assumption}

Assumption \ref{AssInA} is not essential for our results; we make it only for simplicity of the exposition.
Without this assumption, the FCLT can be proved over a finite interval over which $x \in \AA$.
In applications, the fluid limit is likely to hit $\AA$ immediately after the overload begins, and remain in $\AA$ thereafter;
see \S 11.3 in \cite{PeW10a}.


\section{The Main Results} \label{secFCLT}

In preparation for the FCLT, we indicate how the limit is affected by the FTSP
in \S \ref{secFTSPinFCLT}.
We then state the main FCLT and important corollaries in \S \ref{secMain} and \S \ref{secCors}.
We conclude in \S \ref{secR1} by indicating how the results simplify in the special case $r \equiv r_{1,2} = 1$,
where FQR reduces to serving the longer queue.

\subsection{The Role of the FTSP's in the Stochastic Limit}\label{secFTSPinFCLT}

Just as the limiting ODE in \eqn{odeDetails} arising in the FWLLN depends on the FTSP's $D(\gamma, \cdot)$ (through the probability $\pi_{1,2} (x(t))$),
so too the stochastic limit process arising in the FCLT refinement depends on these same FTSP's.
Since the FTSP $D(\gamma, \cdot)$ depending on the state $\gamma$ is a positive recurrent QBD under the assumption that $\gamma \in \AA$,
the stochastic refinement depends on the asymptotic variability of the FTSP.
In particular, since the FTSP $D(\gamma, \cdot)$ is a regenerative process (which can be represented as a QBD whenever the ratio $r$ is rational),
the associated cumulative process obtained by integrating the indicator functions $1_{\{D(\gamma, s) > 0\}}$ obeys a FCLT; i.e.,
\bequ \label{cum}
\hat{C}^n_{QBD} (t; \gamma) \equiv n^{-1/2} \int_{0}^{nt} \left(1_{\{D(\gamma, s) > 0\}} - \pi_{1,2} (\gamma)\right) \, ds \Ra  B(\sigma^2 (\gamma) t)
\eeq
in the functions space $\D$ as $n\tinf$, where $B$ is a standard Brownian motion (BM) for each $\gamma \in \AA$.

The constant $\sigma^2 (\gamma)$ appearing inside the BM on the right in \eqn{cum} is often called the {\em asymptotic variance} (see \cite{A03,GW93,W92})
of the regenerative process $D(\gamma,s)$ (and the function $f$ with $f(D(\gamma,s)) \equiv 1_{\{D(\gamma,s) > 0\}}$).
For each $\gamma \in \AA$, it is defined as the limit
\bes 
\sigma^2(\gamma) \equiv \lim_{t\tinf} \frac{1}{t} Var \left(\int_0^t 1_{\{D(\gamma,s) > 0\}}ds \right).
\ees

In this paper we will be making extensive use of the {\em regenerative structure}; see \cite{A03,GW93} for background.
In our QBD context, the underlying {\em regenerative cycles} can be determined by successive visits of $D(\gamma, \cdot)$ to any fixed state, i.e.,
starting at a transition into the state and ending at the next transition into that state after first leaving that state.
(The next transition into the state after leaving is the beginning of the next cycle; the cycles are closed on the left and open on the right.)
The asymptotic behavior is determined by the random length of a cycle, $\tau(\gamma)$, and either the random integral over a cycle, $\tilde{Y} (\gamma)$,
or the random centered integral over a cycle, $Y(\gamma)$, where
\bes 
\tilde{Y} (\gamma) \equiv  \int_{0}^{\tau(\gamma)} 1_{\{D(\gamma, s) > 0\}} \, ds
\qandq Y(\gamma) \equiv  \int_{0}^{\tau(\gamma)} (1_{\{D(\gamma, s) > 0\}} - \pi_{1,2} (\gamma)) \, ds.
\ees
The key asymptotic quantities here can be expressed in terms of the means of the first two variables and the variance of $Y(\gamma)$ via
\beql{cumCycle}
\pi_{1,2} (\gamma) = \frac{E[\tilde{Y}(\gamma)]}{E[\tau(\gamma)]} \qandq \sigma^2 (\gamma) = \frac{Var(Y(\gamma))}{E[\tau(\gamma)]};
\eeq
see \cite{A03,GW93}.  Of course, $Y(\gamma) = \tilde{Y} (\gamma) - \pi_{1,2} (\gamma) \tau (\gamma)$,
so that $Var(Y(\gamma))$ can be expressed in terms the means, variances and the covariance of the variables $\tau(\gamma)$ and $\tilde{Y}(\gamma)$,
where $0 \le \tilde{Y} (\gamma) \le \tau(\gamma)$ w.p.1.
Here we have strong regularity, with the random variable $\tau (\gamma)$ having a finite moment generating function
and all these quantities being continuous functions of the state $\gamma$, by virtue of Lemma C.5 of \cite{PeW10b}.


\subsection{The FCLT} \label{secMain}

Let $A^n_i (t)$ count the number of class-$i$ customer arrivals,
let $S^n_{i,j} (t)$ count the number of service completions of class-$i$ customers by agents in pool $j$,
an let $U^n_i (t)$ count the number of class-$i$ customers to abandon from queue, all in model $n$ during the time interval $[0,t]$.
Let $D^n_{1,2} (t)$ be the queue-difference process
 in \eqref{QD} and let $Q^n_s (t) \equiv Q^n_1 (t) + Q^n_2 (t)$, all at time $t$.
Let $p_1 \equiv r/(1+r)$ and $p_2 \equiv 1 - p_1 = 1/(1+r)$, where $r \equiv r_{1,2}$.
For $t \ge 0$ and $i,j = 1,2$, let the diffusion-scaled processes be
\bequ \label{DiffScale}
\begin{split}
\hata^n_{i}(t) & \equiv \frac{A^n_{i}(t) - n \lambda_{i}(t)}{\sqrt{n}}, \qquad
\hatu^n_{i}(t) \equiv \frac{U^n_{i}(t) - n \theta_i \int_0^t q_{i}(s) \, ds}{\sqrt{n}}, \\
\hatz^n_{i,j}(t) & \equiv \frac{Z^n_{i,j}(t) - n z_{i,j}(t)}{\sqrt{n}}, \quad
\hats^n_{i,j}(t) \equiv \frac{S^n_{i,j}(t) - n \mu_{i,j} \int_0^t z_{i,j}(s) \, ds}{\sqrt{n}},  \\
\hatq^n_1(t) & \equiv \frac{Q^n_1(t) - n q_1(t)}{\sqrt{n}}, \qquad
\hatq^n_2(t) \equiv \frac{Q^n_2(t) - n q_2(t)}{\sqrt{n}}, \\
\hatq^n_{s}(t) & \equiv \frac{Q^n_{s}(t) - n q_{s}(t)}{\sqrt{n}},  \qquad
\hatd^n(t) \equiv \frac{D^n_{1,2}(t)}{\sqrt{n}} \\ 
\hati^n(t) & \equiv \sqrt{n} \int_{0}^{t} (1_{\{D^n_{1,2} (s) > 0\}} - \pi_{1,2} (x(s))) ds, \quad t \ge 0,
\end{split}
\eeq
where $x \equiv (q_1, q_2, z_{1,2})$ is the customary three-dimensional representation of the fluid limit,
$z_{1,1} \equiv m_1e$, $z_{2,1} = 0e$ $z_{2,2} \equiv m_2e - z_{1,2}$, $q_s \equiv q_1 + q_2$ and $\pi_{1,2} (x(s)) \equiv P(D(x(s), \infty) > 0)$,
with $D(x(s), \infty)$ being a random variable with the steady-state distribution of the FTSP $\{D(x(s), t): t \ge 0\}$ associated with the
fluid limit $x(s)$ at time $s$.


Here is the main result of this paper: the FCLT for the overloaded X model operating under FQR-T.
Since the limit is clearly a Markov process with continuous sample paths, it is by definition a diffusion process.
Most of the rest of the paper is devoted to its proof.
\begin{theorem}{$(${\em FCLT}$)$}\label{thDiffLimit}
If, in addition to Assumptions \ref{AssDiff}--\ref{AssInA},
\bes 
\left( \hatq^n_s(0), \hatz^n_{1,2}(0) \right) \Rightarrow \left( \hatq_s(0), \hatz_{1,2}(0) \right) \in \RR_2 \qasq n\tinf,
\ees
then, for $i,j = 1,2$,
\bequ \label{DiffLimitGen}
\left(\hata^n_i, \hatu^n_i,\hats^n_{i,j}, \hatd^n, \hati^n, \hatq^n_{i}, \hatq^n_{s}, \hatz^n_{i,j}  \right) \Rightarrow
\left(\hata_i, \hatu_i,\hats_{i,j}, \hatd, \hati, \hatq_{i}, \hatq_{s}, \hatz_{i,j}  \right)
\eeq
in $\D_{17}$, where the processes depending on $n$ on the left are defined in \eqref{DiffScale} and the limit process has continuous paths w.p.1.
The initial $10$-dimensional component
$(\hata_i, \hatu_i,\hats_{i,j}, \hatd, \hati)$ is a vector of independent Brownian motions, time scaled by increasing
continuous deterministic functions $($for the first $8$, the fluid limits in the translation terms of \eqref{DiffScale}$)$, with two null components
$\hats_{2,1} \equiv 0e$ and $\hatd \equiv 0e$.  Five components of the limit are determined by the relations
$\hatq_{i} \deq p_i \hatq_{s}$, $\hatz_{2,1} \equiv \hatz_{1,1}\equiv 0e$ and $\hatz_{2,2}\equiv -\hatz_{1,2}$.
Finally, $(\hatq_s, \hatz_{1,2})$ is the unique solution of the following two-dimensional stochastic integral equation:
\bequ \label{sde}
\bsplit
\hatq_s(t) & = \hatq_s(0) + (\mu_{2,2} - \mu_{1,2}) \int_0^t \hatz_{1,2}(s)\, ds  - (p_1\theta_1 + p_2\theta_2) \int_0^t \hatq_s(s)\, ds \\
& \quad   + \hatL_1 (t) -  \hatL_{1,2} (t) -\hats_{1,2} (t) - \hatL_{2,2} (t) - \hats_{2,2} (t),  \\
\hatz_{1,2}(t) & = \hatz_{1,2}(0) - \int_0^t{\left[(\mu_{2,2} - \mu_{1,2})\pi_{1,2}(x(s)) + \mu_{1,2}\right]\hatz_{1,2}(s) \, ds}  \\
& \quad -  \hatL_{1,2} (t) + \hatL_{2,2} (t) + \hatL_2 (t),
\end{split}
\eeq
where, for $i = 1,2$,
\bequ \label{limitComponents}
\bsplit
\hatL_1 & \equiv \hata_1 + \hata_2 - \hatu_1 - \hatu_2 - \hats_{1,1} \deq \{B_1 \left( \gamma_1(t) \right): t \ge 0\}, \\
\hatL_{i,2} & \equiv   \{B_{i,2} (\phi_{i,2}(t)): t \ge 0\},
\quad \hats_{i,2}  \equiv   \{B_{i,3} \left( \gamma_{i,2}(t)\right): t \ge 0\} , \\
\hatL_2     & \equiv   \{B_{2} \left( \gamma_2 (t) \right): t \ge 0\} \qandq
\hati        \equiv  \{ B_{2} \left( \gamma_3 (t) \right): t \ge 0\}, \\
\end{split}
\eeq
with $B_1$, $B_{1,2}$, $B_{2,2}$, $B_{1,3}$, $B_{2,3}$ and $B_2$ being six independent standard BM's, while $\gamma_i$, $\gamma_{i,2}$ and
$\phi_{i,2}$
are strictly increasing continuous deterministic functions.  Specifically,
\bequ \label{gamma}
\bsplit
\gamma_1(t) &\equiv (\lm_1 + \lm_2 + m_1\mu_{1,1}) t + (p_1\theta_1 + p_2\theta_2)\int_0^t q_s(u) \, du \\
\phi_{1,2} (t)   &   \equiv \mu_{1,2} \int_0^t (1 - \pi_{1,2} (x(u))) z_{1,2}(u) \, du,  \\
\phi_{2,2} (t)   &   \equiv \mu_{2,2} \int_0^t \pi_{1,2} (x(u))(m_2 - z_{1,2}(u))\, du,  \\
\gamma_{1,2} (t) &   \equiv \mu_{1,2} \int_0^t \pi_{1,2} (x(u)) z_{1,2}(u) \, du \\
\gamma_{2,2} (t) &   \equiv \mu_{2,2} \int_0^t (1 -\pi_{1,2} (x(u))) (m_2 - z_{1,2}(u))\, du,  \\
\gamma_2(t)      &   \equiv \int_0^t   \psi^2 (x(u)) \sigma^2 (x(u))\, du, \quad
\gamma_3(t)         \equiv \int_0^t   \sigma^2 (x(u))\, du, \\
\end{split}
\eeq
where
\bequ \label{psiDef}
\psi (x(u)) \equiv \mu_{2,2}  (m_2 - z_{1,2} (u)) + \mu_{1,2} z_{1,2} (u), \quad u \ge 0,
\eeq
with $\pi_{1,2} (x(u))$ and $\sigma^2 (x(u))$ being the quantities associated with the FTSP $D(x(u), \cdot)$, defined in \eqref{pi}
and \eqref{cum}, respectively, and characterized in \eqref{cumCycle}.
\end{theorem}

Since the FCLT describes a refinement of the transient behavior of the fluid limit,
it should not be surprising that the limiting
 stochastic process $(\hatq_s, \hatz_{1,2})$ would be difficult to analyze.
On the positive side, we can solve for $\hatz_{1,2}$ in \eqref{sde} without having to simultaneously solve for $\hatq_s$, but we
need $\hatz_{1,2}$ to solve for $\hatq_s$.
An additional complication for $\hatq_s$ is the dependence between the driving Brownian motions for the two processes $\hatq_s$ and
$\hatz_{1,2}$; note that the time-transformed Brownian terms $\hat{L}_{i,2}$ appear in both.

The FCLT shows the impact of system variability on the stochastic limit.
First, and perhaps of greatest interest, there is a Brownian contribution $\hatL_2 \deq B_{2} \left( \gamma_2 (t)\right)$ from the FTSP
appearing in the equation for $\hatz_{1,2}$; note the dependence between $\hatL_2$ and $\hati$.
However,  $(\hatL_2, \hati)$ is independent of all other Brownian terms.
We thus see that the fluctuations about the fixed target ratio $r$ in the queue-difference process \eqref{QD} due to FQR {\em do} have an impact on the stochastic limit.

On the other hand, we see that the stochastic fluctuations associated with external arrivals and abandonments only affect
$\hatq_s$; they have no impact on $\hatz_{1,2}$.  The same is true for the stochastic fluctuations of service facility $1$,
which is always busy, without any sharing.  These fluctuations are captured by the Brownian term
$\hatL_1 \deq B_1 \left( \gamma_1(t) \right)$.
However, as noted above, in distinct contrast, the stochastic fluctuations in the service processes at service facilty $2$ have a more complicated impact,
because they appear in the Brownian driving processes of both equations.

\subsection{Important Corollaries}\label{secCors}

The stochastic limit in the FCLT depends critically on the fluid limit $x$, which typically must be computed numerically,
but an efficient algorithm was developed in \cite{PeW10a}, exploiting the QBD structure of the FTSP $D$ when $r_{1,2}$ is rational.
Since we are mainly interested in the steady state variance of the diffusion limits, and since the stochastic fluctuations
become more significant when the fluid is nearly constant (which happens when it is close to its stationary point)
it is reasonable to initialize the fluid model at this fluid stationary point in order to simplify the expressions in \eqref{sde} and \eqn{gamma}.
We do this in the next corollary.

From an application point of view, the fluid limit is ``more important'' than the refined stochastic limit during the fluid transient period,
since then the changes in the prelimit are of order $O (n)$.
It follows from Theorem \ref{thFluidStat} that after some (relatively short) time,
the fluid stabilizes close to its unique stationary point $x^*_6$ in \eqref{statPt}.
After that happens, the refined stochastic limits become the significant approximation to consider.

When we consider the stochastic refinement of the stationary fluid limit $x^*$, the stochastic limit process becomes
much more tractable:  it is a bivariate Ornstein-Uhlenbeck (BOU) process centered at the origin, as in \cite{A74,W82}.
Consequently, the random vector $(\hatq_s (t), \hatz_{1,2} (t))$
has a bivariate normal distribution with zero means for all $t$, and the associated steady-state random vector $(\hatq_s (\infty), \hatz_{1,2} (\infty))$
can be very useful in applications.  It is characterized by three parameters:  the two variances and the covariance, which we exhibit explicitly in \eqref{covMatrix} below.

For a matrix $M$, let $M^t$ denote its transpose.  The following is the key result for applications.  It gives explicit Gaussian approximations
for the steady-state distributions of all quantities of interest. 

\begin{corollary}{$(${\em FCLT with a stationary fluid}$)$} \label{corDiffLim}
If, in addition to the conditions of Theorem \ref{thDiffLimit}, $x(0) = x^*$ for the stationary point $x^*$ in \eqref{statPt}
so that $x$ is stationary, then the time transformations in \eqref{gamma} simplify by having
$\gamma_i(t) = \xi_i t$, $\gamma_{i,2} (t) = \xi_{i,2} t$, and $\phi_{i, 2} (t) = \eta_{i,2} t$, $i = 1,2$, where
\bequ \label{gamma2}
\bsplit
\xi_1      &   \equiv 2(\lm_1 + \lm_2) - \mu_{1,2} z^{*}_{1,2} - \mu_{2,2} (m_2 - z^{*}_{1,2}),  \\
\xi_{1,2}  &   \equiv \mu_{1,2} \pi_{1,2} (x^{*}) z^{*}_{1,2},   \quad
\xi_{2,2}     \equiv \mu_{2,2}  (1 - \pi_{1,2} (x^{*}) (m_2 - z^{*}_{1,2}),  \\
\eta_{1,2} &   \equiv \mu_{1,2} (1 - \pi_{1,2} (x^{*})) z^{*}_{1,2},   \\
\eta_{2,2} &   \equiv \mu_{2,2} \pi_{1,2} (x^{*})(m_2 - z^{*}_{1,2}),  \\
\xi_2      &   \equiv  \psi^2 (x^*) \sigma^2 (x^{*}) \qandq
\xi_3         \equiv   \sigma^2 (x^{*}), \\
\end{split}
\eeq
for $\sigma^2 (x^{*})$ and $\psi(x^*)$ defined in \eqref{cum} and \eqref{psiDef} with $x(u) = x^{*}$.
Then $(\hatq_s, \hatz_{1,2})$ becomes a BOU process, satisfying the two-dimensional stochastic differential equation $(sde)$
\bequ \label{BOUmatrix}
d \sX = \sM \sX + \sS d \sB,
\eeq
where $\sX \equiv (\hatq_s, \hatz_{1,2})^t$, $\sB  \equiv (B_1, B_2)^t$, with
$B_1$ and $B_2$ being two independent standard BM's, and
\bequ \label{components}
\bsplit
\sM_{1,1} & \equiv -(p_1\theta_1 + p_2\theta_2), \quad \sM_{1,2} \equiv (\mu_{2,2} - \mu_{1,2}), \quad \sM_{2,1}  \equiv 0,\\
\sM_{2,2} & \equiv  \frac{- \mu_{1,2} \mu_{2,2} m_2 z^*_{1,2}}{\mu_{1,2} z^*_{1,2} + \mu_{2,2} (m_2 - z^*_{1,2})} < 0, \\
\sS_{1,1}^2 & \equiv  \xi_1 + \xi_{1,2} + \xi_{2,2} + \eta_{1,2} + \eta_{2,2} =  2 (\lambda_1 + \lambda_2), \\
\sS_{1,2} & \equiv \sS_{2,1} \equiv \eta_{1,2} - \eta_{2,1} = 0, \quad \sS_{2,2}^2 \equiv \xi_2  + \xi_4, \\
\xi_4   & \equiv \eta_{1,2} + \eta_{2,2} = \frac{2\mu_{1,2} \mu_{2,2} z^*_{1,2} (m_2 - z^*_{1,2})}{\mu_{1,2}z^*_{1,2} + (m_2 - z^*_{1,2})\mu_{2,2}}. \\
\end{split}
\eeq
As a consequence,
 $(\hatq_s (t), \hatz_{1,2} (t))$ has a bivariate normal distribution with zero means for each $t$.  The
covariance matrix of the steady-state random vector $(\hatq_s (\infty), \hatz_{1,2} (\infty))$ has elements
\bequ \label{covMatrix}
\bsplit
\sigma^2_{Q_s} (\infty) & \equiv Var(\hatq_s) = \sQ_1 + \sQ_2, \\
& \quad \sQ_1 \equiv \frac{\sS_{1,1}^2}{2 |\sM_{1,1}|} = \left(\frac{\lambda_1 + \lambda_2}{p_1\theta_1 + p_2\theta_2}\right), \\
& \quad \sQ_2 \equiv \frac{\sM_{1,2} \sigma^2_{Q_s,Z_{1,2}} (\infty)}{|\sM_{1,1}|} =
\left(\frac{(\mu_{2,2} - \mu_{1,2}) \sigma^2_{Q_s,Z_{1,2}} (\infty)}{p_1\theta_1 + p_2\theta_2}\right), \\
\sigma^2_{Z_{1,2}} (\infty) &  \equiv \frac{\sS_{2,2}^2}{2 |\sM_{2,2}|} \equiv \sZ_1 + \sZ_2, \\
& \quad \sZ_1 \equiv \frac{\xi_4}{2 |\sM_{2,2}|} = 1 - \frac{z^*_{1,2}}{m_2}, \quad
 \sZ_2 \equiv  \frac{\xi_2}{2 |\sM_{2,2}|} = \frac{\psi^2 (x^*) \sigma^2 (x^{*})}{2 |\sM_{2,2}|}, \\
\sigma^2_{Q_s,Z_{1,2}} (\infty) & \equiv Cov(\hatq_s, \hatz_{1,2}) =  \xi_5 \sigma^2_{Z_{1,2}} (\infty),
\quad \xi_5 \equiv \left(\frac{\sM_{1,2}}{|\sM_{1,1} + \sM_{2,2}|}\right).\\
\end{split}
\eeq
\end{corollary}

\begin{proof}
By the definition of a stationary point, if $x(0) = x^*$ then $x(t) = x^*$ for all $t > 0$ given in \eqref{statPt};
then $\pi_{1,2} (x^*)$ appears in \eqref{piSS}.
The expressions in \eqref{gamma2} follow directly from the expressions in \eqref{gamma}, by replacing the time-dependent fluid quantities
by their stationary counterparts.
The resulting pair of integral equations for $(\hatq_s (t), \hatz_{1,2} (t))$ is known to be equivalent to the BOU sde in \eqref{BOUmatrix}.
The covariance matrix of the stationary distribution, $\Sigma$, is known to satisfy the matrix equation
$\sM \Sigma + \Sigma \sM^t = -\sV$, where $\sV \equiv \sS \sS^t$,
from which \eqref{covMatrix} follows; e.g., see \cite{A74} and \cite{KT81}. Algebra shows that $\xi_4/2|\sM_{2,2}| = (1 - (z^*_{1,2}/m_2))$.
\hfill \qed \end{proof}

\begin{remark} \label{remNull} {\em (when components become null)} {
Notice that the results in Corollary \ref{corDiffLim} simplify greatly with pool-dependent service rates,
i.e., when $\sM_{1,2} \equiv \mu_{2,2} - \mu_{1,2} = 0$.  Then $\sQ_2 = 0$ and $\xi_5 = 0$, so that
$\sigma^2_{Q_s,Z_{1,2}} (\infty) = 0$.
}
\end{remark}

We now see how Theorem \ref{thDiffLimit} simplifies under the condition of pool-dependent service rates
(no longer assuming that $x(0) = x^*$).

\begin{corollary}{$(${\em FCLT with pool-dependent service rates}$)$}\label{corEqRate}
If, in addition to the assumptions of Theorem \ref{thDiffLimit},  $\mu_{2,2} = \mu_{1,2} \equiv \nu$,
then the two diffusion-limit processes $\hatq_s$ and $\hatz_{1,2}$
can both be represented as separate one-dimensional processes, which satisfy the following integral equations
\bes 
\bsplit
\hatq_s(t) &= \hatq_s(0) - \tilde{\eta}_2 \int_0^t{ \hatq_s(s)\, ds} + B_1 \left( \tilde{\gamma}_1(t) \right), \\
\hatz_{1,2}(t) &= \hatz_{1,2}(0) - \nu \int_0^t{\hatz_{1,2}(s)\, ds} + B_2 \left( \tilde{\gamma}_2(t) \right),
\end{split}
\ees
where
\bes 
\bsplit
\tilde{\gamma}_1(t) &\equiv 2(\lm_1 + \lm_2) t + \left( \frac{\tilde{\eta}_1}{\tilde{\eta}_2} - q_s(0) \right)(1-e^{-\tilde{\eta}_2 t}) \\
\tilde{\gamma}_2(t) &\equiv \nu \left(  \int_0^t{[m_2 \pi_{1,2}(x(u)) + z_{1,2}(u)
- 2 \pi_{1,2}(x(u)) z_{1,2}(u)] \, du} \right)  + \gamma_2 (t),
\end{split}
\ees
with $\gamma_2 (t)$ defined in \eqref{gamma},
\bes 
\bsplit
\tilde{\eta}_1 &\equiv \lm_1 + \lm_2 - m_1\mu_{1,1} - m_2\nu, \qquad \tilde{\eta}_2 \equiv p_1\theta_1 + p_2\theta_2, \\
\end{split}
\ees
but $B_1$ and $B_2$ are dependent standard BM's.
\end{corollary}

\proof
It is immediate from the expression for $\hatq_s$ in \eqref{sde} that when $\mu_{1,2} = \mu_{2,2}$
the diffusion process $\hatq_s$ can be analyzed separately from $\hatz_{1,2}$.  Since $q_i = p_i q_s$ and $\mu_{1,2} = \mu_{2,2}$,
it follows from \eqref{odeDetails} that $\dot{q}_s(t)$ satisfies the simple ordinary differential equation
\bes
\dot{q}_s(t) = (\lm_1 + \lm_2 - m_1\mu_{1,1} - m_2\mu_{2,2}) - (p_1\theta_1 + p_2\theta_2) q_s(t) \equiv \tilde{\eta}_1 - \tilde{\eta}_2 q_s(t),
\ees
whose solution is
\bes
q_s(t) = \frac{\tilde{\eta}_1}{\tilde{\eta}_2} + \left( q(0) - \frac{\tilde{\eta}_1}{\tilde{\eta}_2} \right) e^{-\tilde{\eta}_2 t}
\ees
for $\tilde{\eta}_1$ and $\tilde{\eta}_2$ in the statement of the lemma.  Notice that $\tilde{\gamma}_1(t)$ here corresponds to $\gamma_1 (t) + \gamma_{1,2} (t) + \gamma_{2,2} (t)$
in \eqref{gamma}.  Inserting $q_s(t)$ above into $\gamma_1(t)$ in \eqref{gamma} 
gives $\tilde{\gamma}_1(t)$.
Notice that $\tilde{\gamma}_2 (t)$ corresponds to $\phi_{1,2} (t) + \phi_{2,2} (t) + \gamma_2 (t)$ in \eqref{gamma}.
Again substituting yields the conclusion.

\begin{remark} \label{remI} {\em (Equivalence with the single-class model.)} {
If, in addition to the conditions of both Corollaries \ref{corDiffLim} and \ref{corEqRate}, we also have $\theta_1 = \theta_2 \equiv \theta$,
then the diffusion-limit process $\hatq_s$
is the same as the limit obtained for the $M/M/n + M$ model in the efficiency-driven (ED) regime, see \cite{W04}.
That is, $\hatq_s$ is an Ornstein-Uhlenbeck
process with infinitesimal mean equal to $\theta$ and infinitesimal variance $2\lm \equiv 2(\lm_1 + \lm_2$). Thus, its steady-state
distribution is normal with mean zero and variance $\lm / \theta$.
However, $\hat{Z}_{1,2}$ remains somewhat complicated involving $\gamma_2 (t)$ in \eqn{gamma}.}
\end{remark}


\subsection{The Case r = 1:  Longer Queue First (LQF)}\label{secR1}

The most complicated feature in the FWLLN and FCLT asymptotic results in the previous two sections, inhibiting application, is the need to analyze the FTSP.
Specifically, both the approximating fluid model and the stochastic refinement depend critically on the FTSP
$D \equiv D(\gamma) \equiv \{D(\gamma , s) : s \ge 0\}$ at each point $\gamma \in \AA$.
In particular, both limits depend on $D(\gamma)$ through the two functions $\pi_{1,2}(\gamma)$ and $\sigma^2 (\gamma)$.
These two functions can be computed numerically, as indicated above. 
For the stationary fluid point $x^*$, $\pi_{1,2} (x^*)$ is given explicitly in \eqn{piSS}.

However, there is an important special case, itself of practical value, in which the analysis simplifies greatly,
which can provide insight more generally.
When the target queue ratio is $r = 1$, the FTSP $D(\gamma)$ becomes an ordinary {\em birth-and-death} (BD) process for each $\gamma \in \AA$.
Then the quantities $\pi_{1,2} (\gamma)$ and $\sigma^2 (\gamma)$ are both easily expressed.
It turns out that they can be expressed
in terms of the first two moments of the busy-period distributions of two $M/M/1$ queues.  We consider that case now.

We now assume that $r = 1$, and take $\gamma \in \AA$.  In this case, the FTSP
evolves as one BD process when $D(\gamma) > 0$ and evolves as another BD process when $D(\gamma) \le 0$. We call $0$ the boundary state.
Let $\lambda_1 (\gamma)$ denote the constant rate up (away from the boundary)
and let $\mu_1 (\gamma)$ denote the constant rate down (toward the boundary) of $D(\gamma)$ when $D(\gamma) > 0$.
Focusing on the movement relative to the boundary, let $\lambda_2 (\gamma)$ denote the constant rate {\em down} (away from the boundary)
and let $\mu_2 (\gamma)$ denote the constant rate {\em up} (toward the boundary)
of $D(\gamma)$ when $D(\gamma) \le 0$.

Note that we need to analyze $D(\gamma)$ only through the associated stochastic process
\bes 
X(\gamma, t) \equiv 1_{\{D(\gamma,t) > 0\}}, \quad t \ge 0,
\ees
which records which region $D(\gamma, t)$ is in at each time $t$.
The stochastic process $X \equiv X(\gamma) \equiv \{X(\gamma, t): t \ge 0\}$ is
a $\{0, 1\}$-valued process associated with an alternating renewal process.  Let $T_1 (\gamma)$ denote a time interval between
the instant of a state change from state $0$ to state $1$ until the next instant of a state change from state $1$ back to state $0$.
Similarly, let $T_2 (\gamma)$ denote a time interval between
instant of a state change from state $1$ to state $0$ until the next instant of a state change from state $0$ back to state $1$.
The successive times in the alternating renewal process are independent random variables distributed as $T_1 (\gamma)$ and $T_2 (\gamma)$.
The process $X(\gamma)$ is a regenerative process in which the regeneration times can be the
successive instant of a state change from state $0$ to state $1$ until the next instant of the same state change again at a later time.
 The intervals between successive regenerations are
distributed as $T_1 (\gamma) + T_2 (\gamma)$.

Now observe that $T_i (\gamma)$ is distributed as a busy period in an $M/M/1$ queue with arrival rate
$\lambda_i (\gamma)$ and service rate $\mu_i (\gamma)$, $i = 1,2$.
In this context, the condition $\gamma \in \AA$ is equivalent to $\lambda_i (\gamma) < \mu_i (\gamma)$, $i = 1,2$.
Under this condition, $T_i (\gamma)$ is known to have
a finite moment generating function with a positive radius of convergence, so that all moments of $T_i (\gamma)$ are finite.
Let
\bequ \label{cum2}
m_i (\gamma) \equiv 1/\mu_i (\gamma) \qandq \rho_i (\gamma) \equiv \lambda_i (\gamma)/\mu_i (\gamma), \quad i = 1,2.
\eeq
Then, from basic $M/M/1$ theory, we have
\bequ \label{cum3}
E[T_i (\gamma)] = \frac{m_i (\gamma)}{1 - \rho_i (\gamma)} \qandq  E[T_i (\gamma)^2] = \frac{2 m_i (\gamma)^2}{(1 - \rho_i (\gamma))^3}.
\eeq

Finally, we are interested in the cumulative process associated with $X(\gamma)$,
\bes
C(\gamma, t) \equiv \int_{0}^{t} X(\gamma, s) \, ds \equiv  \int_{0}^{t} 1_{\{D(\gamma,s) > 0\}} \, ds, \quad t \ge 0.
\ees
We can apply \eqn{cumCycle} to obtain the following result.

\begin{theorem}{$($the FTSP when $r=1$ $)$}\label{thr1}
When $r = 1$ and $\gamma \in \AA$, the FTSP becomes a recurrent BD process.
Hence the key FTSP quantities can be expressed directly in terms of the four BD rates
$\lambda_i (\gamma)$ and $\mu_i (\gamma)$ via
\bequ \label{cum4}
\pi_{1,2} (\gamma) = \frac{E[T_1 (\gamma)]}{E[T_1 (\gamma)] + E[T_2 (\gamma)]}, \quad
 \sigma^2 (\gamma) = \frac{Var(T_1 (\gamma))}{E[T_1 (\gamma)] + E[T_2 (\gamma)]}
\eeq
for $E[T_i (\gamma)]$ and $E[T_i (\gamma)^2]$ in \eqref{cum3} and \eqref{cum2}, $i = 1,2$.
\end{theorem}

In the more general QBD setting arising with $r \not= 1$, the analysis is more complicated,
because the excursions of $\int_{0}^{t} 1_{\{D(\gamma,s) > 0\}}$ above and below $0$ depend on the
entering and exit states from level $0$; thus these excursions are not simply independent.  Theorem \ref{thr1} can be the basis for
heuristic extensions to non-Markovian models in which the arrival, service and abandonment processes are non-Markovian.
We may then exploit approximations for the busy period in $GI/GI/1$ queues, e.g., \cite{AW95} and \cite{SN01}.

\section{Proof of Theorem \ref{thDiffLimit}} \label{secDiffProof}

First observe that the assumed convergence in $\RR_2$ at time $0$ is actually equivalent to
the full convergence in $\RR_{17}$ of the process in \eqn{DiffLimitGen} at time $0$ because of Assumption \ref{AssInitial}.
Our proof has four main steps:  The first step is to exploit SSC results established in \cite{PeW10b}.
In particular, we first give an
asymptotically equivalent three-dimensional representation of $X^n_6$ (without any scaling) involving rate-$1$ Poisson processes.
Then we observe that the essential dimension is actually two (when scaling by $\sqrt{n}$) because the queue lengths are asymptotically in the fixed ratio.
Thus we deduce that it is sufficient to directly prove convergence of the $2$-dimensional process $(\hatq^n_{s}, \hatz^n_{1,2})$.

The second step is to facilitate application of the continuous mapping theorem by showing that an essential mapping is continuous.
The third step is to construct appropriate martingale representations, allowing application of the continuous mapping theorem.
The fourth and final hardest step is to show that the driving stochastic terms in this martingale representation converge to the specified limits.
This final step uses a new result of independent interest, Theorem \ref{lmTvFCLT}, the generalization of the classical FCLT for cumulative processes
in \eqn{cum} to the case where the QBD parameters at time $t$ are given by the fluid limit $x(t)$, which in general is time-varying.

\subsection{Representation and SSC}

Following common practice, as reviewed in \S 2 of \cite{PTW07}, we represent the
 processes $A^n_i (t)$, $S^n_{i,j} (t)$ and $U^n_i (t)$ introduced at the beginning of \S \ref{secMain}
 in
terms of mutually independent rate-$1$ Poisson processes; let
\bes 
\bsplit
A^{n}_i (t)     & \equiv N^{a}_i (\lambda^{n}_i t), \\
S^{n}_{i,j} (t) & \equiv N^{s}_{i,j} \left(\mu_{i,j} \int_{0}^{t} Z^{n}_{i,j} (s) \, ds\right) \qandq S^n \equiv \sum_{j=1}^{2}\sum_{i=1}^{2}S^n_{i,j}, \\
U^{n}_{i} (t)   & \equiv N^{u}_i \left(\theta_i \int_{0}^{t} Q^{n}_{i} (s) \, ds\right), \quad t \ge 0,
\end{split}
\ees
where $N^{a}_i$, $N^{s}_{i,j}$ and $N^{u}_i$ for $i = 1,2; j = 1,2$ are eight mutually independent rate-$1$ Poisson processes.
Theorem 5.1 of \cite{PeW10b} gives a representation of the CTMC in terms of these processes.
Corollaries 6.1-6.3 plus Theorem 6.4 of \cite{PeW10b} then establish state space collapse (SSC) results yielding an
asymptotically equivalent three-dimensional representation of $X^n_6$ involving these mutually independent rate-$1$ Poisson processes
plus two others.  Since we exploit that representation, we state it here.
Directly, the representation of $Z^n_{1,2}$ below keeps it in the interval $[0, m^n_2]$.  However, the representation directly allows
the queue lengths $Q^n_i$ to become negative.  The results in \cite{PeW10b} show that the occurrence (anywhere in a
a bounded interval) is asymptotically negligible.
Recall that $d_{J_1}$ denotes the Skorohod $J_1$ metric.

\begin{lemma}{$($Representation via SSC of the service process$)$}\label{thSSC1}
Under the assumptions in Theorem \ref{th1},
$d_{J_1}(X^n_6, X^{n,*}_6) \Ra 0$ in $\D_6$ as $n \tinf$, with the three determining components of $X^{n,*}_6$ in \eqref{Xn*},
i.e., in $X^n$ in \eqref{Xn}, being represented via
\bes 
\bsplit
Z^{n}_{1,2} (t)  & \equiv Z^{n}_{1,2} (0)  +  \int_{0}^{t} 1_{\{D^n_{1,2}(s-) > 0\}} \, dS^n_{2,2} (s)
- \int_{0}^{t} 1_{\{D^n_{1,2}(s-) \le 0\}} \, dS^n_{1,2} (s)  \\
& \deq Z^{n}_{1,2} (0) + N^{s}_{2,2} \left(\mu_{2,2} \int_{0}^{t} 1_{\{D^n_{1,2}(s) > 0\}} (m^n_2 - Z^{n}_{1,2}(s)) \, ds\right) \\
& \quad - N^{s}_{1,2} \left(\mu_{1,2} \int_{0}^{t} 1_{\{D^n_{1,2}(s) \le 0\}}  Z^{n}_{1,2} (s) \, ds\right), \\
\end{split}
\ees
\bes 
\bsplit
Q^{n}_{1} (t) & \equiv Q^{n}_{1} (0)  +  A^n_1 (t) - \int_{0}^{t} 1_{\{D^n_{1,2}(s-) > 0\}} \, dS^n (s)
- \int_{0}^{t} 1_{\{D^n_{1,2}(s-) \le 0\}} \, dS^n_{1,1} (s) - U^n_1 (t)  \\
& \deq Q^{n}_{1} (0) + N^{a}_{1} (\lambda^n_1 t) -  N^{s}_{1, 1} (\mu_{1,1} m^n_{1} t)
-  N^{s,2}_{1,2} \left(\mu_{1,2} \int_{0}^{t} 1_{\{D^n_{1,2}(s) > 0\}} Z^{n}_{1,2} (s)) \, ds\right) \\
& \quad - N^{s}_{2,2} \left(\mu_{2,2} \int_{0}^{t} 1_{\{D^n_{1,2}(s) > 0\}} (m^n_2 - Z^{n}_{1,2} (s)) \, ds\right)
- N^u_1 \left(\theta_{1} \int_{0}^{t} Q^{n}_{1} (s) \, ds\right), \\
\end{split}
\ees
\bes 
\bsplit
Q^{n}_{2} (t) & \equiv Q^{n}_{2} (0)  +  A^n_2 (t) - \int_{0}^{t} 1_{\{D^n_{1,2}(s-) \le 0\}} \, dS^n_{2,2} (s)
- \int_{0}^{t} 1_{\{D^n_{1,2}(s-) \le 0\}} \, dS^n_{1,2} (s) - U^n_2 (t) \\
& \deq Q^{n}_{2} (0) + N^{a}_{2} (\lambda^n_2 t)
-  N^{s,2}_{2,2} \left(\mu_{2,2} \int_{0}^{t} 1_{\{D^n_{1,2}(s) \le 0\}} (m^n_2 - Z^{n}_{1,2} (s)) \, ds\right) \\
& \quad - N^{s}_{1,2} \left(\mu_{1,2} \int_{0}^{t} 1_{\{D^n_{1,2}(s) \le 0\}} Z^{n}_{1,2} (s) \, ds\right)
- N^u_2 \left(\theta_{2} \int_{0}^{t} Q^{n}_{2} (s) \, ds\right).
\end{split}
\ees
where $N^{s,2}_{1,2}$ and $N^{s,2}_{2,2}$ are two additional rate-$1$ Poisson processes, independent of the others.
\end{lemma}
The representation in Lemma \ref{thSSC1} provides important simplification, but it also shows the difficulty in proving heavy traffic limit theorems;
the integrals contain the indicator functions depending on $D^n_{1,2}$.
We now show that the essential dimension can be reduced from three to two when we introduce scaling.
The next result follows from Corollary 4.1 of \cite{PeW10b}.

\begin{lemma}{$($SSC to two dimensions$)$}\label{thSSC2}
Under the conditions of Theorem {\em \ref{th1}}, the essential dimension can be reduced from $3$ established in Lemma {\em \ref{thSSC1}} to $2$, because
$d_{J_1}(Q^n_1, rQ^n_2)/a_n \Ra 0$ in $\D([0, \delta))$ for $\delta$ in Theorem {\em \ref{th1}} whenever $a_n/\log n \ra \infty$ as $n \tinf$.
If $x \in \AA$ over an interval $[t_1, t_2)$, $0< t_1 < t_2 \le \infty$, then the conclusion holds in $\D((t_1, t_2))$.
\end{lemma}

Due to Assumption \ref{AssInA} and Lemma \ref{thSSC2},
it is sufficient to directly prove convergence of the $2$-dimensional process $(\hatq^n_{s}, \hatz^n_{1,2})$;
the more general $16$-dimensional limit in \eqn{DiffLimitGen} can be obtained
as a byproduct of the analysis, and in particular, $\hatq_i \deq p_i \hatq_s$, $i = 1,2$.

\subsection{A Continuous Mapping}

As in \cite{PTW07}, our proof exploits the continuous mapping theorem. However, in our case, the stochastic processes
describing the evolution of the system (the queue length and service processes) cannot be expressed directly as
a continuous mapping of the primitive processes.
We next establish the continuity of the mapping that we will eventually apply.

\begin{lemma}{\em $($Continuity of the two-dimensional integral representation$)$} \label{lmCont}
Consider the two-dimensional integral representation
\bes 
\bsplit
x_1(t) &= b_1 + y_1(t) + \alpha_2 \int_0^t{x_2(s) \, ds} + \alpha_1 \int_0^t{x_1(s) \, ds} \\
x_2(t) &= b_2 + y_2(t) + \int_0^t{g(s)x_2(s) \, ds}
\end{split}
\ees
where $g : \RR \ra \RR$ satisfies $g(0) = 0$ and is Lipschitz continuous with a Lipschitz constant $c_g$.
That integral representation has a unique solution $(x_1, x_2)$, so that the integral representation constitutes a function
$f : \D_2 \times \RR_2 \ra \D_2$ mapping $(y_1, y_2, b_1, b_2)$ into $(x_1, x_2) \equiv f(y_1, y_2, b_1, b_2)$.
In addition, the function $f$ is a continuous mapping from $\D_2 \times \RR_2$ to $\D_2$.
Moreover, if $y_2$ is continuous then $x_2$ is continuous. If both $y_1$ and $y_2$ are
continuous, then $x_1$ is also continuous.
\end{lemma}

\begin{proof}
By the conditions on the function $g$ we have for all $T \ge 0$
\bes
\| g \|_T \le g(0) + \| g(u) - g(0) \|_T \le g(0) + c_g T = c_g T.
\ees
Note that $x_2$ does not depend on $x_1$, hence we can prove the lemma iteratively by first showing that
the function $f_2 : \D \times \RR$ mapping $(y_2, b_2)$ into $x_2 \equiv f_2(y_2, b_2)$ is continuous, and then use this
result to show that the function $f_1 : \D_2 \times \RR$ mapping $(y_1, x_2, b_1)$ into $x_1 \equiv f_1(y_1, x_2, b_1)$
is continuous.

To show that $f_2$ is continuous we use Theorem 2.11 in \cite{TW09} with $h(x_2(u), u) \equiv g(u)x_2(u)$.
For that purpose, choose $T > 0$ and
let $\lm$ be a homeomorphism on $[0, T]$ with strictly positive derivative $\dot{\lm}$. Then, for every $\varphi_1, \varphi_2 \in \D$
\bes
\bsplit
& \int_0^t{ | g(u)\varphi_1(u) - g(\lm(u))\varphi_2(\lm(u)) |\, du } \\
&\le \int_0^t{ | g(u)\varphi_1(u) - g(u)\varphi_2(\lm(u)) | \, du} + \int_0^t{ | g(u)\varphi_2(\lm(u)) - g(\lm(u))\varphi_2(\lm(u)) | \, du } \\
&\le \| g \|_{T} \int_0^t{ | \varphi_1(u) - \varphi_2(\lm(u)) | \, du } + \| \varphi_2 \|_{T} \int_0^t{ | g(u) - g(\lm(u)) | \, du} \\
&\le \| g \|_{T} \int_0^t{ | \varphi_1(u) - \varphi_2(\lm(u)) | \, du } + c_g T \| \varphi_2 \|_{T} \| \lm - e \|_{T} \\
&= c_1 \| \lm - e \|_T + c_2 \int_0^t{ | \varphi_1(u) - \varphi_2(\lm(u)) | \, du }.
\end{split}
\ees
where $c_1 \equiv c_g T \| \varphi_2 \|_T$ and $c_2 \equiv \| g \|_T$.

For $x_1 = f_1(y_1, x_2, b_1)$ we can apply Theorem 4.1 in \cite{PTW07} with input $y \equiv y_1 + \alpha_2\int_0^t{x_2(u) \, du}$.
It follows from Theorem 2.11 in \cite{TW09} that if $y_2$ is continuous then so is $x_2$. If, in addition, $y_1$ is continuous,
then $y$ is continuous and, by Theorem 4.1 in \cite{PTW07}, so is $x_1$.
\hfill \qed \end{proof}

\subsection{Martingale Representations}

As in Theorem 6.3 of \cite{PeW10b}, we next apply the representation in Lemmas \ref{thSSC1} and \ref{thSSC2} to obtain martingale representations for $\hat{Q}^n_s$
and $\hat{Z}^n_{1,2}$, but now we are interested in the FCLT instead of the FWLLN.
We exploit martingale representations for the counting processes appearing in lemma \ref{thSSC1} constructed from
the rate-$1$ Poisson processes $N^a_i$, $N^s_{i,2}$, $N^{s,2}_{i,2}$ and $N^u_i$, $i = 1,2$, in particular,
\begin{align} \label{martgs}
\allowdisplaybreaks[4]
M^n_{1,1}(t) & \equiv N^s_{1,1}(m^n_1\mu_{1,1} t) - m^n_1\mu_{1,1}t, \nonumber \\
M^n_{1,2}(t) & \equiv N^s_{1,2} \left(\mu_{1,2} \int_{0}^{t} 1_{\{D^n_{1,2}(s) \le 0\}}  Z^{n}_{1,2} (s) \, ds \right)
- \mu_{1,2} \int_{0}^{t} 1_{\{D^n_{1,2}(s) \le 0\}}  Z^{n}_{1,2} (s) \, ds,  \nonumber \\
M^n_{2,2}(t) & \equiv N^s_{2,2} \left(\mu_{2,2} \int_{0}^{t} 1_{\{D^n_{1,2}(s) > 0\}} (m^n_2 - Z^{n}_{1,2}(s)) \, ds \right) \nonumber \\
& \quad - \mu_{2,2} \int_{0}^{t} 1_{\{D^n_{1,2}(s) > 0\}} (m^n_2 - Z^{n}_{1,2}(s)) \, ds, \nonumber \\
M^n_{1,3}(t) & \equiv N^{s,2}_{1,2} \left(\mu_{1,2} \int_{0}^{t} 1_{\{D^n_{1,2}(s) > 0\}}  Z^{n}_{1,2} (s) \, ds \right)
- \mu_{1,2} \int_{0}^{t} 1_{\{D^n_{1,2}(s) > 0\}}  Z^{n}_{1,2} (s) \, ds,  \\
M^n_{2,3}(t) & \equiv N^{s,2}_{2,2} \left(\mu_{2,2} \int_{0}^{t} 1_{\{D^n_{1,2}(s) \le 0\}} (m^n_2 - Z^{n}_{1,2}(s)) \, ds \right) \nonumber \\
& \quad - \mu_{2,2} \int_{0}^{t} 1_{\{D^n_{1,2}(s) \le 0\}} (m^n_2 - Z^{n}_{1,2}(s)) \, ds, \nonumber \\
M^n_{a_i}(t) & \equiv N^a_i(\lm^n_i t) - \lm^n_i t, \quad i = 1,2, \nonumber \\
M^n_{u_i}(t) & \equiv N^u_i \left( \theta_i \int_0^t{Q^n_i(s) ds} \right) - \theta_i \int_0^t{Q^n_i(s) ds}, \quad i = 1,2. \nonumber
\end{align}

\begin{lemma}{$($martingale representation for $\hat{Q}^n_s)$}\label{thMGq}
\bes 
\bsplit
\hatq^n_s(t) &= \hatq^n_s(0) + (\mu_{2,2} - \mu_{1,2}) \int_0^t{\hatz^n_{1,2}(s)\, ds} - (p_1\theta_1 + p_2\theta_2)\int_0^t{ \hatq^n_s(s)\, ds}  \\
&\quad + \hat{M}^n_s(t) + o_P(1) \qasq n \ra \infty,
\end{split}
\ees
where $\hat{M}^n_s \equiv M^n_s/\sqrt{n}$ for the martingale
\bequ \label{MartgSum}
M^n_s(t) \equiv \sum_{i=1}^{2}{M^n_{a_i}}(t) - \sum_{i=1}^{2}{M^n_{u_i}(t)} - \sum_{i=1}^{2}{M^n_{i,2}(t)} - \sum_{i=1}^{2}{M^n_{i,3}(t)} - M^n_{1,1}(t).
\eeq
with respect to the natural filtration.
\end{lemma}

\begin{proof}
By Theorem \ref{thSSC1},
\bes
\bsplit
Q_s^{n} (t) &=  Q^n_s(0) + (\lm_1^n + \lm_2^n)t - m^n_1\mu_{1,1}t - \mu_{1,2}\int_0^t{Z^n_{1,2}(s)\, ds} - \mu_{2,2}\int_0^t{Z^n_{2,2}(s)\, ds} \\
&\quad - \theta_1\int_0^t{Q^n_1(s)\, ds} - \theta_2\int_0^t{Q^n_2(s) \, ds} + M^n_s(t),
\end{split}
\ees
for $M^n_s(t)$ in \eqn{MartgSum}.
Observe that the indicator functions in the representation of $X^n$ in Lemma \ref{thSSC1}
do not appear in the representation of $Q_s^n(t)$. That simplifies the analysis.

From \eqref{odeDetails} it follows that $q_s \equiv q_1 + q_2$, the fluid counterpart of $Q^n_s$,
evolves according to the integral equation:
\bes
\bsplit
q_s(t) &= q_s(0) + (\lm_1 + \lm_2)t - \mu_{1,1}m_1 t - \mu_{1,2}\int_0^t{z_{1,2}(u)\, du} - \mu_{2,2} \int_0^t{z_{2,2}(u)\, du} \\
&\quad - \theta_1 \int_0^t{ q_1(u) \, du} - \theta_2 \int_0^t{q_2(u)\, du},
\end{split}
\ees
so that, substituting $q_1$ with $p_1 q_s(u)$ and $q_2(u)$ with $p_2 q_s(u)$,
\bes
\bsplit
q_s(t) &= q_s(0) + (\lm_1 + \lm_2)t - \mu_{1,1}m_1 t - \mu_{2,2}m_2 t \\
& \quad + (\mu_{2,2} - \mu_{1,2})\int_0^t{z_{1,2}(u)\, du} - (p_1\theta_1 + p_2\theta_2) \int_0^t{q_s(u) \, du}.
\end{split}
\ees
Then, by centering about $n q_s$ and dividing by $\sqrt{n}$ as in \eqref{DiffScale}, we have
\beql{center}
\bsplit
\hatq_s^n(t) &= \hatq^n_s(0) + \frac{[(\lm^n_1 + \lm^n_2) - n(\lm_1 + \lm_2)]t}{\sqrt{n}} - \frac{\mu_{1,1}(m^n_1 - nm_1)t}{\sqrt{n}} \\
&\quad - \frac{\mu_{1,2}\int_0^t{(Z^n_{1,2}(s) - nz_{1,2}(s))\, ds}}{\sqrt{n}} -
\frac{\mu_{2,2}\int_0^t{(Z^n_{2,2}(s) - nz_{2,2}(s))\, ds}}{\sqrt{n}} \\
&\quad - \frac{\theta_1\int_0^t{(Q^n_1(s) - nq_1(s))\, ds}}{\sqrt{n}} - \frac{\theta_2\int_0^t{(Q^n_2(s) - nq_2(s))\, ds}}{\sqrt{n}}
 + \frac{M^n_s(t)}{\sqrt{n}}.
\end{split}
\eeq
By Assumption \ref{AssDiff}, the second and third terms in the expression above converge to zero.
By Corollary 6.2 and Theorem 6.4 in \cite{PeW10b}, $n^{-1/2} \|Z^n_{2,2} - (m^n_2 - Z^n_{1,2})\| \Ra 0$ in $\D$ as $n \tinf$
so that $z_{2,2} = m_2 - z_{1,2}$. Also, $(m_2^n - n m_2)/\sqrt{n} \ra 0$ as $n \tinf$ by Assumption \ref{AssDiff}.
Hence,
\bes 
\bsplit
\hatq_s^n & = \hatq^n_s(0) + (\mu_{2,2} - \mu_{1,2}) \int_0^t{\hatz^n_{1,2}(s)\, ds} \\
& \quad - \theta_1\int_0^t{\hatq^n_1(s)\, ds} - \theta_2\int_0^t{\hatq^n_2(s)\, ds} + \hat{M}^n_s(t) + o_{P} (1).
\end{split}
\ees
Define
\bes
\bsplit
\hat{Q}^n_{a,s} (t) &\equiv \hatq^n_s(0) + (\mu_{2,2} - \mu_{1,2}) \int_0^t{\hatz^n_{1,2}(s)\, ds} - p_1\theta_1\int_0^t{ \hatq^n_s(s)\, ds} \\
&\quad -p_2\theta_2\int_0^t{\hatq^n_s(s)\, ds} + \hat{M}^n_s(t)
\end{split}
\ees
By applying the SSC result in Lemma \ref{thSSC2}, we conclude that
$\| \hatq_s^n - \hat{Q}^n_{a,s} \|_T \Rightarrow 0$ in $\D$ as $n \tinf$ for any $T > 0$. That completes the proof.
\hfill \qed \end{proof}

We now turn to the process $Z^n_{1,2}$.
\begin{lemma}{$($martingale representation for $\hat{Z}^n_{1,2})$}\label{thMGz}
\bequ \label{repZdiff}
\bsplit
\hatz^n_{1,2} (t)
& = \hatz^n_{1,2} (0) - \int_0^t{\left[(\mu_{2,2} - \mu_{1,2})\pi_{1,2}(x(s)) + \mu_{1,2}\right]\hatz^n_{1,2}(s) \, ds}  \\
& \quad + \hat{L}^n + \hat{M}^n_{Z} + o (1),
\end{split}
\eeq
where $\hat{L}^n \equiv L^n/\sqrt{n}$, $\hat{M}^n_{Z} \equiv M^n_{Z}/\sqrt{n}$,
\bequ \label{Y_Z}
\bsplit
L^n (t) & \equiv  \int_{0}^{t} [1_{\{D^n_{1,2}(s) > 0\}} - \pi_{1,2} (x(s))] \Psi^n (s) \, ds, \\
\Psi^n (s) & \equiv \mu_{2,2}  (m^n_2 - Z^n_{1,2} (s)) + \mu_{1,2} Z^{n}_{1,2} (s)
\end{split}
\eeq
and $M^n_Z$ is the martingale
\bequ \label{Zmartg}
M^n_Z(t) \equiv M^n_{2,2} (t) - M^n_{1,2}(t).
\eeq
with respect to the natural filtration, where $M^n_{2,2}$ and $M^n_{1,2}$ the martingales in {\em \eqn{martgs}}.
\end{lemma}

\begin{proof}
We start by
rewriting the representation of $Z^n_{1,2}$ in Lemma \ref{thSSC1} as
\bes 
\bsplit
Z^{n}_{1,2} (t) & = Z^n_{1,2} (0) -\mu_{1,2} \int_{0}^{t} (1 - \pi_{1,2} (x(s)))  Z^{n}_{1,2} (s) \, ds \\
& \quad + \mu_{2,2} \int_{0}^{t} \pi_{1,2} (x(s)) (m^n_2 - Z^{n}_{1,2} (s)) \, ds + L^n + M^n_{Z}.
\end{split}
\ees
To achieve the diffusion-scaled process, we center $Z^n_{1,2}$ about $n z_{1,2}$ and divide by $\sqrt{n}$,
where, by \eqn{odeDetails}, the fluid limit $z_{1,2}$ satisfies the equation
\bes
\bsplit
z_{1,2}(t) & = z_{1,2} (0) + \mu_{2,2} \int_0^t {\pi_{1,2}(x(s))(m_2 - z_{1,2}(s)) \, ds} \\
& \quad - \mu_{1,2} \int_0^t{(1 - \pi_{1,2}(x(s))) z_{1,2}(s) \, ds}.
\end{split}
\ees
We get the representation \eqn{repZdiff} with
the $o(1)$ term replacing the deterministic term
$[(m^n_2 - n m_2)\int_{0}^{t} \pi_{1,2} (x(s)) \, ds]/\sqrt{n} \le (m^n_2 - n m_2) t/\sqrt{n}$,
which converges to zero by Assumption \ref{AssDiff}.
\hfill \qed \end{proof}

\subsection{Convergence of Stochastic Driving Terms}

Given the representations in Lemmas \ref{thMGq} and \ref{thMGz}, we can complete the proof
of the convergence of $(\hat{Q}^n_s, \hat{Z}^n_{1,2})$ in Theorem \ref{thDiffLimit} by establishing convergence of the driving terms and applying the continuous mapping theorem
with the mapping in Lemma \ref{lmCont}, i.e., with the following lemma, proved in the next section.
We add an extra process, $\hat{I}^n$, also defined in \eqn{DiffScale}, which is closely related to
$\hat{L}^n$, but not directly needed to treat $(\hat{Q}^n_s, \hat{Z}^n_{1,2})$.

\begin{lemma}{$($convergence of driving terms$)$}\label{lmLim}
Under the assumptions of Theorem {\em \ref{thDiffLimit}},
\bequ \label{drivingLim}
(\hat{M}^n_s, \hat{M}^n_Z, \hat{L}^n, \hat{I}^n) \Rightarrow (\hat{M}_s, \hat{M}_Z, \hat{L}_2, \hat{I}) \qinq \D_4,
\eeq
where
\begin{align*}
\hat{M}_s (t) & \equiv B_1 \left( \gamma_1(t) \right) - B_{1,2} \left( \gamma_{1,2}(t) \right) - B_{2,2} \left( \gamma_{2,2}(t) \right) \\
& \quad - B_{1,3} \left( \phi_{1,2}(t) \right) - B_{2,3} \left( \phi_{2,2}(t) \right), \\
\hat{M}_Z (t) & \equiv B_{2,2} \left( \phi_{2,2} (t) \right) -  B_{1,2} (\phi_{1,2}(t)), \\
\hat{L}_2 (t)   & \equiv  B_2 (\gamma_2 (t)) \qandq  \hat{I} (t)   \equiv B_2 (\gamma_3 (t)), \quad t \ge 0,
\end{align*}
$B_1$, $B_{1,2}$, $B_{2,2}$, $B_{1,3}$, $B_{2,3}$ and $B_2$ are independent standard BM's
as in the statement of Theorem {\em \ref{thDiffLimit}} and
$\gamma_1(t)$, $\gamma_2 (t)$, $\gamma_3 (t)$, $\gamma_{1,2}(t)$, $\gamma_{2,2}(t)$, $\phi_{1,2}(t)$ and $\phi_{2,2}(t)$ are the increasing continuous functions
defined in \eqref{gamma}.
\end{lemma}

\subsection{Overall Proof of Theorem \ref{thDiffLimit}}

We prove convergence of $(\hat{Q}^n, \hat{Z}^n_{1,2})$ by applying the continuous mapping theorem with the continuous function in Lemma \ref{lmCont},
exploiting the representations in Lemmas \ref{thMGq} and \ref{thMGz}
and the convergence established in Lemma \ref{lmLim}.
In applying Lemma \ref{lmCont}, we rely heavily on Theorem 7.1 in \cite{PeW10a}, which establishes that
$\pi_{1,2}(\cdot)$ is locally Lipschitz continuous in $\AA$ as a function of the fluid state $x$
and is thus Lipschitz continuous over compact sets.
Moreover, $x (\cdot)$ is itself Lipschitz continuous, as a function of the time argument $s$ by Corollary 5.1 in \cite{PeW10b}.
It follows that $\pi_{1,2} (x(s))$ is Lipschitz continuous as a function of the time argument $s$ as well
(using Assumption \ref{AssInA} implying that $x$ lies entirely in $\AA$).
Thus the proof of Theorem \ref{thDiffLimit} is complete with the exception of the proof of Lemma \ref{lmLim}.
The next four sections are devoted to that proof.

\section{Proof of Lemma \ref{lmLim}}\label{secProofLim}

This section is devoted to proving Lemma \ref{lmLim}.
In \S \ref{secFirstTwo} we apply standard arguments to establish the convergence of the first two martingale terms
$(\hat{M}^n_s, \hat{M}^n_Z)$.  In preparation for
treating the last two terms, in \S \ref{secSupport} we state two key results that we will use; they are proved in the following three sections.
In \S \ref{secLastTwo} we establish convergence of the last two terms $(\hat{L}^n, \hat{I}^n)$.
Finally, in \S \ref{secIndep} we establish joint convergence of all four terms by
proving asymptotic independence of the last two terms from the first two terms.

\subsection{The First Two Terms in \eqref{drivingLim}}\label{secFirstTwo}

We start by establishing convergence of the first two terms in Lemma \ref{lmLim}, the two martingale terms.
\begin{lemma}\label{lmFirstTwo}  There is joint convergence of the martingale processes
\bes 
(\hat{M}^n_s, \hat{M}^n_Z) \Rightarrow (\hat{M}_s, \hat{M}_Z) \qinq \D_2,
\ees
where the processes are defined in {\em \eqn{MartgSum}}, {\em \eqn{Zmartg}} and Lemma {\em \ref{lmLim}}.
\end{lemma}

\begin{proof}
Let
\beas
\hat{M}^n_S (t) & = & \left( \hat{M}^n_{1,1}(t), \hat{M}^n_{1,2}(t), \hat{M}^n_{2,2}(t), \hat{M}^n_{1,3}(t), \hat{M}^n_{2,3}(t) \right) \qinq \D_5, \\
\hat{M}^n_A (t) &= & \left( \hat{M}_{a_1}(t), \hat{M}_{a_2}(t) \right) \qandq
\hat{M}^n_U (t)  =  \left( \hat{M}^n_{u_1}(t), \hat{M}^n_{u_2}(t) \right) \qinq \D_2
\eeas
for the martingale processes in \eqref{martgs}.
To compress the notation, for $x \in \D_k$ and $t \in [0,\infty)^k$, let $x(t) \equiv (x_1(t_1), x_2(t_2), \dots, x_k(t_k))$.
We start by proving that
\bequ \label{MartgLim}
\left( \hat{M}^n_A (t), \hat{M}^n_S (t), \hat{M}^n_U (t) \right) \Rightarrow
\left( B_A(\lm t), B_S (\phi(t)), B_U \left( \theta \int_{0}^{t} q(s) \, ds
 \right) \right) \qinq \D_9
\eeq
as $n \tinf$, where
\bequ \label{MartgLimDef}
\bsplit
\phi (t) &\equiv (\phi_1 (t), \phi_2 (t), \phi_3 (t), \phi_4 (t), \phi_5 (t)), \\
\phi_1 (t) &\equiv m_1 \mu_{1,1} (t), \quad \phi_2 (t) \equiv \phi_{1,2} (t), \quad \phi_3 (t) \equiv \phi_{2,2} (t), \\
\phi_4 (t) & \equiv \gamma_{1,2} (t), \quad \phi_5 (t) \equiv \gamma_{2,2} (t),
\end{split}
\eeq
for $\phi_{i,2}$ and $\gamma_{i,2}$ defined in \eqn{gamma}.
Here $B_A(t)$, $B_S(t)$ and $B_U(t)$ are vectors of  independent standard Brownian motions. Using our compressed notation,
we have $\lm t \equiv (\lm_1 t, \lm_2 t)$ and
$\theta q(s) \equiv (\theta_1 q_1(s), \theta_2 q_2(s))$. For example, $B_A(\lambda t) = \left( B_{A_1} (\lm_1 t), B_{A_2} (\lm_2 t) \right)$,
and similarly for $B_S(\cdot)$ and $B_U (\cdot)$.

To prove \eqref{MartgLim}, we apply the FCLT for Poisson processes.
For the Poisson processes in Lemma \ref{thSSC1}, let
\bes 
\bsplit
\tilde{M}^n_{a_i} &= \frac{N^a_i(n t) - nt}{\sqrt{n}}, \quad \tilde{M}^n_{i,j} = \frac{N^s_{i,j}(nt) - nt}{\sqrt{n}} \quad \qandq \\
\tilde{M}^n_{u_i} &= \frac{N^u_i(nt) - nt}{\sqrt{n}}, \qquad i = 1,2; j = 1,2,3.
\end{split}
\ees
Let $\tilde{M}^n_A(t)$, $\tilde{M}^n_S(t)$ and $\tilde{M}^n_{U}(t)$ be the corresponding vector-valued processes.
By the independence of all the unit-rate Poisson processes $N^a_i(\cdot)$, $N^s_{i,j}(\cdot)$ and $N^u_i(\cdot)$, and the FCLT for a Poisson process, the following
joint convergence holds:
\bequ \label{Mlim}
\left( \tilde{M}^n_A(t), \tilde{M}^n_S(t), \tilde{M}^n_U (t) \right) \Rightarrow \left( \tilde{B}_A(t), \tilde{B}_S(t), \tilde{B}_U(t) \right)
\quad \mbox{in $\D_9$ \ as $n \tinf$},
\eeq
where $\tilde{B}_A$, $\tilde{B}_S$ and $\tilde{B}_U$ are, respectively, $2$-dimensional, $5$-dimensional and $2$-dimensional
independent Brownian motions; see Theorem 4.2 and \S 9.1 in \cite{PTW07}.

We now introduce random time changes.
Let
\begin{eqnarray}\label{timetr}
\Phi^n_{A,i}(t) & \equiv  &  n^{-1} \lm^n_i t, \qquad
\Phi^n_{S,1}(t)  \equiv  n^{-1}\mu_{1,1} m^n_1 t, \nonumber \\
\Phi^n_{S,2}(t) & \equiv & n^{-1} \mu_{1,2} \int_{0}^{t} 1_{\{D^n_{1,2}(s) \le 0\}}  Z^{n}_{1,2} (s) \, ds,  \nonumber \\
\Phi^n_{S,3}(t) & \equiv & n^{-1} \mu_{2,2} \int_{0}^{t} 1_{\{D^n_{1,2}(s) > 0\}} (m^n_2 - Z^{n}_{1,2}(s)) \, ds, \nonumber  \\
\Phi^n_{S,4}(t) & \equiv & n^{-1} \mu_{1,2} \int_{0}^{t} 1_{\{D^n_{1,2}(s) > 0\}}  Z^{n}_{1,2} (s) \, ds, \nonumber  \\
\Phi^n_{S,5}(t) & \equiv & n^{-1} \mu_{2,2} \int_{0}^{t} 1_{\{D^n_{1,2}(s) \le 0\}} (m^n_2 - Z^{n}_{1,2}(s)) \, ds, \nonumber  \\
\Phi^n_{U,i}(t) &\equiv & n^{-1} \theta_i \int_0^t{Q^n_i(s) \, ds}, \quad i,j = 1,2.
\end{eqnarray}
By Assumption \ref{AssDiff} on the arrival rates, $\Phi^n_{A_i} \ra \lm_i e$ in $\D$, $i = 1,2$.
From the initial conditions in the statement of Theorem \ref{thDiffLimit}, the fluid limit and the continuity of the integral mapping, it follows that
$\Phi^n_{S,i} \Rightarrow \phi_i$, $1 \le i \le 5$,  and $\Phi^n_{U,i} (t) \Rightarrow \theta_i \int_0^t{q_i(s)\, ds}$
in $\D$ as $n \tinf$.

Let $\Phi^n_{A}(t)$, $\Phi^n_{S}(t)$ and $\Phi^n_{U}(t)$ be the corresponding vector-valued processes. By Theorem 11.4.5 of \cite{W02},
these limits hold jointly, yielding
\bequ \label{PhiLim}
\left( \Phi^n_{A}(t), \Phi^n_{S}(t), \Phi^n_{U}(t) \right) \Rightarrow
\left( \lm t,\, \phi (t),\, \theta \int_0^t{q(s)\, ds} \right) \qinq \D_{9}
\eeq
as $n \tinf$.  By Theorem 11.4.5 of \cite{W02}, the limits in \eqn{Mlim} and \eqn{PhiLim} also hold jointly.
By definition,
\bes
\left( \hat{M}^n_A(t), \hat{M}^n_S(t), \hat{M}^n_{U}(t) \right) =
\left( \tilde{M}^n_{A} \Big(\Phi^n_A(t) \Big), \tilde{M}^n_S \Big(\Phi^n_S (t) \Big), \tilde{M}^n_{U} \Big(\Phi^n_U(t) \Big) \right).
\ees
Thus, the convergence in \eqref{MartgLim} follows from the continuity of the composition mapping at continuous limits, Theorem 13.2.1 in \cite{W02}.
Finally, the conclusion of the lemma itself then follows from the definition of $\hat{M}^n_s$ and $\hat{M}^n_Z$ in \eqref{MartgSum} and \eqn{Zmartg},
and the continuity of addition under continuous limits, e.g., Corollary 12.7.1 in \cite{W02}.
\hfill \qed \end{proof}

\subsection{Key Supporting Results for the Last Two Terms}\label{secSupport}

In \S \ref{secFTSPinFCLT} we indicated that the stochastic limit will depend on the FCLT for cumulative processes
associated with the FTSP, as stated in \eqn{cum}.  As indicated in \S 6 of \cite{PeW10a}, the FTSP
with fixed state $\gamma$ is a QBD; its parameters (transition rates) are given explicitly in (13)-(16) of \cite{PeW10b}.
Since the FTSP $D(\gamma, \cdot)$ is a QBD for each state $\gamma$, it is a relatively simple
regenerative stochastic process for each state $\gamma$, assuming that $\gamma$ makes the QBD positive recurrent.
However, in our application, the fluid state is {\em not} fixed at $\gamma$, but is instead given by the fluid limit $x(t)$,
which is a function of time $t$.  That means that the parameters of the FTSP are actually time-varying.
By Assumption \ref{AssInA},
the FTSP with fluid state $x(t)$ is a positive recurrent QBD for all states $x(t)$ considered.
Moreover, by Lemma C.5 of \cite{PeW10b}, the infinitesimal generator and the asymptotic variance of the QBD
are continuous functions of the underlying state $x(t)$.  Since the essential matrix structure (e.g., the dimension of the matrices) of the QBD's
depends only on the rational ratio parameter $r_{1,2}$, and thus does not change, the QBD is characterized by only finitely many parameters.
As a consequence, we can establish a variant of the FCLT in \eqn{cum}, allowing the FTSP to have a time-varying state.

In our remaining proof of Lemma \ref{lmLim}, in particular for Lemma \ref{lmNewFrozen} below,
we will want to generalize the state of the QBD.
The parameters of the QBD depend on the fluid state $\gamma \equiv (q_1, q_2, z_{1,2})$, but
also on the rest of the QBD parameters,
in particular, also upon $\zeta \equiv (\lambda_i, m_j; i, j = 1,2)$.  In order to establish Lemma \ref{lmNewFrozen} below, we will want to allow the parameters
$\lambda_i$ and $m_j$ to vary, because they vary with $n$ in the many-server heavy-traffic scaling in Assumption \ref{AssDiff}.
The QBD also depends on the other model parameters $\theta_i$ and $\mu_{i,j}$, but they are fixed, so we do not include them.
Thus, we will consider the more general ``full'' {\em parameter state function} $\eta \equiv (\zeta, \gamma)$ for
$\eta \equiv \eta (t)$ and $\gamma \equiv \gamma (t)$ above,
which we understand to be an element of the functions space $\sD$.
We obtain a conventional stationary QBD model for each full parameter state $\eta (t)$.

Now we will establish a FCLT for
\bequ \label{cumTV1}
\hat{C}^n (t; \eta) \equiv  n^{-1/2} \int_{0}^{nt} \left(1_{\{D(\eta(s/n), s) > 0\}} - \pi_{1,2} (\eta (s/n))\right) \, ds, \quad t \ge 0,
\eeq
where the state function $\eta$ is an element of $\D$ and $D(\eta(0), 0)$ is some fixed finite initial value.
Note that in the special case of a constant parameter state function, with $\eta (t) = \gamma$, $0 \le t \le T$,
this new process reduces to the previous one in \S \ref{secFTSPinFCLT}; i.e.,
\bes 
\hat{C}^n (t; \eta) = \hat{C}^n_{QBD} (t; \gamma), \quad 0 \le t \le T.
\ees
for $\hat{C}^n_{QBD} (t; \gamma)$  in \eqn{cum}.

However, more generally, the process $\hat{C}^n (t; \eta)$ in \eqn{cumTV1} is more complicated, so that
the new FCLT is by no means immediate.
The non-constant function $\eta$ makes the new process $\{D (\eta(s/n),s): s \ge 0\}$ appearing in the integrand of \eqn{cumTV1} neither a QBD nor a regenerative process.
Nevertheless, we establish the following
 generalization of the FCLT in \eqn{cum}.  The proof is given in \S \ref{secProofQBDFCLT}.
\begin{theorem}{$($FCLT for FTSP with time-varying parameter state$)$}\label{lmTvFCLT}
Consider the FTSP $D$ as a function of its parameter state function $\eta$ specified above,
where $\eta$ is a function in $\sD$.
Suppose that the QBD $D(\eta (t), \cdot)$ is positive recurrent for all $\eta (t)$, $0 \le t \le T$.
Then
\bes 
\hat{C}^n (\cdot; \eta) \Ra \hat{C} (\cdot; \eta) \qinq \D ([0,T]) \qasq n \ra \infty,
\ees
where $\hat{C}^n (\cdot; \eta)$ is given in {\em \eqn{cumTV1}} and
\bes 
\hat{C} (t; \eta) \equiv  B\left(\int_{0}^{t} \sigma^2 (\eta (u)) \, du\right), \quad t \ge 0,
\ees
with $B$ being a standard BM and, for each $u$,
$\sigma^2 (\eta (u))$ is the asymptotic variance of the cumulative process with constant full parameter state $\eta(u)$,
as in {\em \eqn{cum}}-{\em \eqn{cumCycle}}.
\end{theorem}

For Lemma \ref{lmNewFrozen} below,
we will also want to extend the FCLT in Theorem \ref{lmTvFCLT} to full parameter state functions that are suitably near a given
deterministic one.
For that purpose, we use the following elementary corollary to Theorem \ref{lmTvFCLT} and its proof.  (Also see \S 6 of \cite{PeW10a}
and \S C.3 of\cite{PeW10b}.
We use the Prohorov metric $d_{\sP, T} (Y_1, Y_2)$ characterizing convergence in distribution in $\D([0,T])$; see p. 77 of \cite{W02}.
We say that a parameter-state function $\eta$ is positive recurrent if the associated FTSP $D(\eta, \cdot)$ is positive recurrent.

\begin{corollary}{$($continuity of the FCLT for the FTSP with time-varying parameter state$)$}\label{lmTvFCLTcont}
Consider the FTSP $D$ as a function of its parameter state function $\eta$ specified above,
where the parameter state function $\eta$ is a positive-recurrent element of $\D$.
For all $\ep > 0$ and $T>0$,
there exists $\delta > 0$ such that, if $\eta'$ is a parametric state function satisfying $\|\eta - \eta'\|_T < \delta$,
then $\eta'$ is positive recurrent for all $t$ in $[0,T]$ and
$d_{\sP, T} (\hat{C} (\cdot; \eta),\hat{C} (\cdot, \eta')) < \ep$
 where $\hat{C} (\cdot; \eta)$ is the limit process associated with $D(\eta, \cdot)$ in Theorem \ref{lmTvFCLT}. 
\end{corollary}

\begin{proof}
We exploit the criterion for
recurrence in terms of the drift rates given in \eqn{posrec}.  The drift rates $\delta_{+} (\eta)$
and $\delta_{-} (\eta)$ for constant $\eta$ in the regions $\{s: D(\eta, s) > 0\}$
and $\{s: D(\eta, s) \le 0\}$, respectively, are linear functions of the components of the vector $\eta$.
We can thus express the drifts as the inner products $\delta_{\pm} (\eta) = a_{\pm} \cdot \eta$,
where $a_{+}$ and $a_{-}$ are constant vectors.
Hence, if $|\eta - \eta'| \le \ep$, then $| \delta_{\pm} (\eta) -  \delta_{\pm} (\eta')| \le \ep (|a_{\pm}|\cdot 1)$,
where here $1$ is a vector of $1's$ of the appropriate dimension.
This property for constant parameter states extends immediately to more general state functions in $\D$
using the uniform norm; i.e., if $\|\eta - \eta'\|_T \le \ep$, then
$\| \delta_{\pm} (\eta) -  \delta_{\pm} (\eta')\|_T \le \ep (|a_{\pm}|\cdot 1)$.
Thus, for any positive recurrent state function $\eta$, there exists $\ep > 0$ such that
$\delta_{+} (\eta') < 0$ and $\delta_{-} (\eta') > 0$ if $\| \eta - \eta\|_T < \ep$,
implying that $\eta'$ is also positive recurrent.
\hfill \qed \end{proof}

In Lemma \ref{lmNewFrozen} below, we will apply Corollary \ref{lmTvFCLTcont} to random state functions
$\tilde{\eta}_n$ which converge weakly to $\eta$ as $n \ra \infty$,
i.e., for which $\tilde{\eta}_n \Ra \eta$ in $\D$ as $n \ra \infty$.
To do so,
we need to connect the queue-difference processes $D^n_{1,2}$ appearing
in $\hat{I}^n$ in \eqn{DiffScale} to the FTSP.  We do that via the associated {\em frozen processes}, introduced in \S A.1 of \cite{PeW10b}.
The frozen process
$\{D^{n}_{f}(X^n (t), s): s \ge 0\}$ corresponds to the queue-difference process $D^n_{1,2}$ starting at time $t$,
conditioned on the state $X^n(t)$ at time $t$ under the assumption that the transition rates are fixed (``frozen'')
at the rates associated with the initial state $X^n (t)$.
A key property, for applying Theorem \ref{lmTvFCLT} and Corollary \ref{lmTvFCLTcont} above, is that the frozen process can be represented as the FTSP with modified parameters.
To express the connection, we write the frozen process and the FTSP as functions of the parameters $(\lambda_i, m_j, \gamma, s)$.
As in equation (74) of \cite{PeW10b}, we have the representation
\bequ \label{frozenFTSP}
\{D^{n}_{f}(\lambda^n_i, m^n_j, X^n(t), s): s \ge 0\} \deq \{D(\lambda^n_i/n, m^n_j/n, X^n (t)/n, n s): s \ge 0\},
\eeq
where $D^n_f$ on the left of \eqref{frozenFTSP} is the frozen process described above, and $D$ on the right of \eqref{frozenFTSP} is the FTSP.

Like the queue-difference process, the frozen process has $O(n)$ transition rates, whereas the FTSP has $O(1)$
transition rates, because of the time scaling in \eqref{expanded}. Thus the time variable $s$ on the right in \eqn{frozenFTSP} is scaled by $n$.

However, we need to construct a process that is made up of different frozen processes over different subintervals.  Thus,
for each $n \ge 1$, we will construct a process that is a different frozen process over each successive interval of length
$1/n$, but identical to the queue-difference process at each interval endpoint.  In particular, we will construct the overall frozen process by setting
\bequ \label{NewFrozen1}
\tilde{D}^{n}_{f} (t) \equiv D^n_f (X^n((k-1)/n), t - (k-1)/n),  \quad \frac{k-1}{n} \le t < \frac{k}{n},
\eeq
$0 \le t \le T$, where $D^n_f$ is the frozen process defined above.  That is, we use a different frozen state and thus frozen process
for each interval $[(k-1)/n, k/n)$ in $[0,T]$.
As a consequence, the frozen process state for the process $\tilde{D}^{n}_{f}$ as a function of $t$ is thus
\bequ \label{newFrozen2}
X^n_f (t) \equiv X^n (\lfloor nt \rfloor/n), \quad 0 \le t \le T.
\eeq

As a consequence of \eqn{frozenFTSP}-\eqn{newFrozen2} above, we can simply write
\bequ \label{frozenFTSP2}
\{\tilde{D}^{n}_{f}(t): t \ge 0\} \deq \{D(\tilde{\eta}_n (t), n t): t \ge 0\},
\eeq
where $\tilde{\eta}_n$ is a random full parameter state function with the special parameter function given in \eqn{frozenFTSP} above, with the frozen state
at time $t$ given by \eqn{newFrozen2}.
Corollary \ref{lmTvFCLTcont} is relevant because, by virtue of Assumption \ref{AssDiff} and Theorem \ref{th1}, for each $T  > 0$, we have
\bes 
\|\tilde{\eta}_n - \eta\|_T \Ra 0 \qasq n \ra \infty,
\ees
where $\eta$ has fixed components $\lambda_i$, $m_j$ and $x(t)$, $t \ge 0$.

Hence, the FCLT for fixed positive recurrent state function $\eta$,
which holds by Theorem \ref{lmTvFCLT}, also holds with $\eta$ replaced by $\tilde{\eta}_n$ by virtue of Corollary \ref{lmTvFCLTcont}.
However, it remains to show that the newly constructed frozen processes approximate the queue-difference processes suitably well.
For that, we will use a special coupling construction, similar to the coupling constructions used in \cite{PeW10b}.
The following result is proved in \S \ref{secFrozenPf}.
\begin{lemma}\label{lmNewFrozen}
For each $n$, we can construct the new frozen processes $\tilde{D}^n_f$ defined by {\em \eqn{frozenFTSP}}-{\em \eqn{frozenFTSP2}}
on the same underlying probability space with the queue-difference processes
 $D^n_{1,2}$ so that, $\Delta^n \Ra 0$ in $\D$ as $n \ra \infty$, where
\bequ \label{NewFrozen12}
\Delta^n (t) \equiv \sqrt{n} \int_{0}^{t} \left(1_{\{D^n_{1,2} (s) > 0\}}  - 1_{\{\tilde{D}^n_{f} (s) > 0\}}\right) \, ds, \quad t \ge 0.
\eeq
\end{lemma}

\subsection{The Last Two Terms in \eqref{drivingLim}}\label{secLastTwo}

We now establish joint convergence of the last two terms in Lemma \ref{lmLim}.
\begin{lemma}\label{lmLastTwo}  There is joint convergence of the last two terms in Lemma {\em \ref{lmLim}}, i.e.,
\bes 
(\hat{L}^n, \hat{I}^n) \Rightarrow (\hat{L}_2, \hat{I}) \qinq \D_2,
\ees
where the converging processes $\hat{L}^n$ and $\hat{I}^n$ are defined, respectively, in {\em \eqn{Y_Z}} and {\em \eqn{DiffScale}}, while
the vector limit process is $(\hat{L}_2 (t), \hat{I} (t)) \equiv (B_2 (\gamma_2 (t), B_2 (\gamma_3 (t))$ for $B_2$ a standard Brownian motion
and $(\gamma_2 (t), \gamma_3 (t))$ in {\em \eqn{gamma}},
as in {\em \eqn{limitComponents}}.
\end{lemma}

\begin{proof}
We start by considering just $\hat{I}^n$.
We make a change of variables in \eqn{DiffScale} to get
\bequ \label{In82}
\hat{I}^n (t) \equiv \frac{1}{\sqrt{n}} \int_{0}^{n t } [1_{\{D^{n}_{1,2}(s/n) > 0\}} - \pi_{1,2} (x(s/n))]  \, ds, \quad 0 \le t \le T.
\eeq
From either the original representation of $\hat{I}^n$ in \eqn{DiffScale} or the equivalent alternative expression in \eqn{In82},
 the main line of the proof should be evident:  We show that the
 time-scaled
 queue-difference process $D^{n}_{1,2}(s/n)$ in \eqn{In82} is asymptotically equivalent to the
 scaled FTSP $D(x(s/n), s)$, making the expression in \eqn{In82} be essentially of the form of $\hat{C}^n$ in 
 \eqn{cumTV1}.
 If we could just directly make that substitution, then the desired limit $\hat{I}^n \Ra \hat{I}$ would be
 an immediate consequence of Theorem \ref{lmTvFCLT}.
 However, the desired substitution is only valid asymptotically.
 We actually achieve the desired approximation by the FTSP indirectly by approximating the queue-difference process
 and applying Lemma \ref{lmNewFrozen} and Corollary \ref{lmTvFCLTcont} in addition to Theorem \ref{lmTvFCLT}.
 In particular, we can write
\begin{eqnarray}
&& \frac{1}{\sqrt{n}} \int_{0}^{n t } [1_{\{D^{n}_{1,2}(s/n) > 0\}} - \pi_{1,2} (x(s/n))]  \, ds = \sqrt{n} \int_{0}^{t} [1_{\{D^{n}_{1,2}(s) > 0\}} - \pi_{1,2} (x(s))]  \, ds \nonumber \\
&& \quad = \sqrt{n} \int_{0}^{t} \left(1_{\{D^n_{1,2} (s) > 0\}}  - 1_{\{\tilde{D}^n_{f} (s) > 0\}}\right) \, ds \nonumber  \\
&& \quad \quad + \frac{1}{\sqrt{n}}\int_{0}^{n t } 1_{\{\tilde{D}^n_{f} (s/n) > 0\}} - \pi_{1,2} (x(s/n))]  \, ds. \nonumber
\end{eqnarray}
We then apply Lemma \ref{lmNewFrozen} to the first component in the RHS of the equality, and Corollary \ref{lmTvFCLTcont} to the second component, using \eqn{frozenFTSP2}.

Having established the limit for $\hat{I}^n$, we turn to $\hat{L}^n$.
From \eqn{Y_Z}, we know that $L^n$ differs from $I^n$ by having the extra term $\Psi^n$ in the integrand.
However, by the FWLLN, Theorem \ref{th1}, $\Psi^n \Ra \psi$ as $n \ra \infty$,
where $\psi (t) \equiv \mu_{2,2} (m_2 - z_{1,2} (t)) + \mu_{1,2} z_{1,2} (t)$.
Hence, we
 first write
\bes 
\hat{L}^n_{1} (t) \equiv  \sqrt{n} \int_{0}^{t} [1_{\{D^n_{1,2}(s) > 0\}} - \pi_{1,2} (x(s))] \psi (s) \, ds, \quad t \ge 0.
\ees
Since $\psi$ is continuous, we can approximate it uniformly closely by a piecewise constant function, with all discontinuities occurring
at multiples of a small positive $\ep$.
Hence, by approximation, we can assume without loss of generality that
\bes 
\hat{L}^n_1 (t) \equiv \sum_{j = 1}^{\lfloor t/\ep \rfloor + 1} \psi_j \hat{I}^n_j (t),
\ees
where $\psi_j$ is a constant for each $j$ and $\hat{I}^n_j(t)$ has the structure of $\hat{I}^n$
over the subinterval $[(j-1)\ep, j \ep)$ and is $0$ outside that interval.
Hence, we have convergence of $\hat{L}^n_1$, jointly with $\hat{I}^n$, by essentially the same argument
as for $\hat{I}^n$.  Finally, we can write
\bes 
\|\hat{L}^n - \hat{L}^n_1\|_T \le \| \bar{\Psi}^n - \psi \|_T \|\hat{I}^n\|_T \Ra 0 \qasq n \ra \infty,
\ees
because $\| \bar{\Psi}^n - \psi \|_T \Ra 0$ and $\hat{I}^n \Ra \hat{I}$ as $n \ra \infty$, so that $\|\hat{I}^n\|_T  \Ra \|\hat{I}\|_T$,
where $\|\hat{I}\|_T$ is finite by the continuous mapping theorem with the function $\| \cdot\|_T$.
Hence, we obtain the claimed joint limit for $(\hat{L}^n , \hat{I}^n )$.
\hfill \qed \end{proof}

\subsection{Joint Convergence in Lemma \ref{lmLim}}\label{secIndep}

To complete the proof of Lemma \ref{lmLim}, it remains to show
that the two limits established in Lemmas \ref{lmFirstTwo} and
\ref{lmLastTwo} actually hold jointly. The two separate limits
directly imply the associated tightness, which in turn implies
the tightness for the sequence of four-dimensional processes.
Thus, to prove convergence it suffices to show that the limits
of all converging subsequences coincide. We uniquely
characterize the joint limit by showing that the
 limit for every convergent subsequence of
the sequence $\{(\hat{M}^n_s, \hat{M}^n_Z, \hat{L}^n, \hat{I}^n)\}$ must be of the form
$(\hat{M}_s, \hat{M}_Z, \hat{L}, \hat{I})$, where $(\hat{M}_s, \hat{M}_Z)$ is independent of
$(\hat{L}, \hat{I})$, having distributions as determined above.  Thus it suffices to establish the following
lemma.

\begin{lemma}{$($independent limits$)$}\label{lmAsymIndep}
The limits $(\hat{M}_s, \hat{M}_Z)$ and $(\hat{L}, \hat{I})$
for every convergent subsequence of the sequence
$\{(\hat{M}^n_s, \hat{M}^n_Z, \hat{L}^n, \hat{I}^n)\}$ are
independent.
\end{lemma}

In order to prove Lemma \ref{lmAsymIndep}, we use the following lemma.
\begin{lemma}{$($basis for independent limits$)$}\label{lmAsymIndep2}
If $\hat{D}^n \Ra 0e$, $\hat{I}^n \Ra \hat{I}$ and
$\hat{V}^n \Ra \hat{V}$ as $n \ra \infty$ for random vectors $(\hat{I}^n, \hat{D}^n, \hat{V}^n)$ in $\D_3$, where
$\hat{D}^n = f(\hat{V}^n)$ for some function $f$ and
$P(\hat{I}^n \in B| \hat{V}^n = v) = P(\hat{I}^n \in B| \hat{D}^n = f(v))$
for all Borel sets $B$ almost surely with respect to $dP(\hat{V}^n = v)$,
then $\hat{I}$ is independent of $\hat{V}$.
\end{lemma}

\begin{proof}  Let $g_i$ be a continuous bounded real-valued function on $\D$ for $i = 1, 2 , 3$.
By the
assumptions above,
\begin{eqnarray}\label{indep1}
E[g_1 (\hati^n) g_2 (\hatd^n) g_3 (\hatv^n)] & = & E[E[g_1 (\hati^n) g_2 (\hatd^n)|\hatv^n] g_3 (\hatv^n)] \nonumber \\
& = & E[E[g_1 (\hati^n) g_2 (\hatd^n)|\hatd^n] g_3 (\hatv^n)] \nonumber \\
& = & E[E[g_1 (\hati^n) |\hatd^n] g_2 (\hatd^n) g_3 (\hatv^n)].
\end{eqnarray}
Since $\hati^n \Ra
\hati$ and $\hatd^n \Ra 0e$, we also have $(\hati^n, \hatd^n)
\Ra (\hati, 0e)$ by Theorem 11.4.5 of \cite{W02}. Thus, $E[g_1
(\hati^n) g_2 (\hatd^n)] \Ra E[g_1 (\hati) g_2 (\hatd)] = g_2
(0e) E[g_1 (\hati)]$ as $n\tinf$, so that $\hati^n$ is
asymptotically independent of $\hatd^n$ and
\beq
E[g_1 (\hati^n) |\hatd^n] g_2 (\hatd^n) \Ra E[g_1 (\hati)] g_2 (0e) \in \RR \qasq n \ra \infty.
\eeqno
By Theorem 11.4.5 of \cite{W02} once again,
\beq
(E[g_1 (\hati^n) |\hatd^n] g_2 (\hatd^n), \hatv^n) \Ra (E[g_1 (\hati)] g_2 (0e) ,\hatv) \qinq \RR \times \D  \qasq n \ra \infty,
\eeqno
so that, applying the continuous mapping theorem with the function $h: \RR \times \D \ra \RR$ defined by $h(x,y) \equiv x g_3(y)$,
we obtain
\beql{indep2}
E[g_1 (\hati^n) |\hatd^n] g_2 (\hatd^n) g_3(\hatv^n) \Ra E[g_1 (\hati)] g_2 (0e) g_3 (\hatv) \qinq \RR.
\eeq
Since the random variables in \eqn{indep2} are bounded, we can apply the bounded convergence theorem, combined with \eqn{indep1} and \eqn{indep2}, to get
\beq
E[g_1 (\hati^n) g_2 (\hatd^n) g_3 (\hatv^n)] \ra E[g_1 (\hati)] g_2 (0e) E[g_3 (\hatv)] \qasq n \ra \infty.
\eeqno
From the special case $g_2 \equiv 1e$,
$E[g_1 (\hati^n) g_3 (\hatv^n)] \ra E[g_1 (\hati)] E[g_3 (\hatv)]$ as $n\tinf$.
Since the product $g_1 g_3$ is a continuous bounded real-valued function, we also have
$E[g_1 (\hati^n) g_3 (\hatv^n)] \ra E[g_1 (\hati) g_3 (\hatv)] \qasq n\tinf$.
Hence, $E[g_1 (\hati) g_3 (\hatv)] = E[g_1 (\hati)] E[g_3 (\hatv)]$ for all continuous bounded real-valued functions $g_1$ and $g_3$, so that
$\hati$ is independent of $\hatv$.
\hfill \qed \end{proof}

\paragraph{Proof of Lemma \ref{lmAsymIndep}}
We show that the conditions of Lemma \ref{lmAsymIndep2}
are satisfied in our case.
For that, we rely strongly on the SSC result in Corollary 4.1 of \cite{PeW10b}.
We first observe that, for each $n$, the stochastic process $\{D^n_{1,2} (t): t \ge 0\}$, and thus also the stochastic processes
$\{1_{\{D^n_{1,2} (t) > 0\}}: t \ge 0\}$ and $\hati^n$
in \eqn{DiffScale} are directly functions of $\hatd^n$ in \eqn{DiffScale}.
Thus, for each $n$, the conditional distribution of $\hati^n$ in $\D$ conditional on
$\hatv^n \equiv \left(\hata^n_i, \hatu^n_i,\hats^n_{i,j}, \hatd^n, \hatq^n_{i}, \hatq^n_{s}, \hatz^n_{i,j}  \right)$ in \eqn{DiffScale}
coincides with the conditional distribution of $\hati^n$ in $\D$ conditional on $\hat{D}^n$.
Moreover,  $\hatd^n$ is scaled in the same way as the other processes in $\hatv^n$ in \eqn{DiffScale}.  However, Theorem 4.5 (iii)
and its corollary 4.1, both from \cite{PeW10b}, imply that $\hatd^n \Ra 0e$.
Thus all the conditions of Lemma \ref{lmAsymIndep2} are satisfied, and the statement of the Lemma follows.
\hfill \qed 

\section{Proof of Theorem \ref{lmTvFCLT}}\label{secProofQBDFCLT}

First, if the parameter state function $\eta$ is a constant function, with $\eta (t) = \gamma$, $0 \le t \le T$,
then $\hat{C}^n (t; \eta) = \hat{C}^n_{QBD} (t,\gamma)$ for $\hat{C}^n (t; \eta)$ in \eqn{cumTV1} and $\hat{C}^n_{QBD} (t,\gamma)$ in \eqn{cum},
as noted in \S \ref{secSupport}.  Moreover, if the QBD $D(\gamma, \cdot)$ is positive recurrent,
then the conclusion in Theorem \ref{lmTvFCLT} reduces to the standard FCLT for a cumulative process in \eqn{cum}.
To consider more general time-varying parameter state functions $\eta \equiv \{\eta(u) : 0 \le u \le T\}$,
we require that $\eta$ be positive recurrent where, as before,
we say that a state function $\eta$ is positive recurrent if the associated FTSP
$D(\eta (t), \cdot)$ is positive recurrent for all $t$, $0 \le t \le T$.

Next, we observe that the conclusion in Theorem \ref{lmTvFCLT}
is also valid for all positive-recurrent
piecewise constant parameter state functions, where we include the condition that
there be only finitely many discontinuities in each bounded interval.
Let $\sD_{pc}$ be the subspace of $\sD$ containing all such piecewise constant functions.
To see that the conclusion holds for each positive recurrent $\eta \in \sD_{pc}$, note that,
because of the time scaling, each subinterval $[a,b)$ of length $O(1)$ for the state function $\eta$
corresponds to an interval of length $O(n)$ for the stochastic process $\{D(\eta(s/n), s): s \ge 0\}$, which has transition rates of $O(1)$.
Moreover, the convergence on each successive interval implies that the initial distributions converge on the next interval.
Hence, the initial conditions on each subinterval do not alter the limit. Thus, the separate subintervals can be treated separately, as if
we were considering the first case of a constant parameter state function.

Intuitively, it should be evident that the result extends to positive recurrent state functions $\eta$ in $\D$
because each such function is the uniform limit over bounded intervals of piecewise-constant state functions;
see p. 393 of \cite{W02}.  However, a complete proof for this seemingly minor extension seems quite complicated.
The remaining proof will be based on a series of lemmas, which are proved in the next section.

First, we exploit Corollary \ref{lmTvFCLTcont} showing that the subset of positive recurrent state functions in $\D$
is an open subset.
With Corollary \ref{lmTvFCLTcont}, we then exploit the continuity of QBD's established in Lemma C.5 of \cite{PeW10b} to
complete the proof.
We complete the proof in several steps, requiring further lemmas.
In doing so, we will exploit frozen processes to simplify the argument.  As before, we use a coupling construction
to show that they serve as suitable asymptotic approximations.

Here we consider a modification of the process $\hat{C}^n$ in \eqn{cumTV1},
having a parameter state that is frozen over each successive cycle, where as before a cycle is the period between successive visits to
a fixed state.  As remarked before, in the case of a constant parameter state function $\eta$, with $\eta (t) = \gamma$, $0 \le t \le T$,
these are the regeneration cycles associated with the regenerative process $D(\gamma, \cdot)$, as in \S \ref{secFTSPinFCLT},
but here we have a more general case. For each $n$, let $\hat{C}^n_f$ denote this modification of $\hat{C}^n$, having a parameter state that is frozen over each successive cycle.
We use a coupling construction to show that it suffices to consider
$\hat{C}^n_f$ in order to establish the desired convergence of $\hat{C}^n$ in \eqn{cumTV1}.

  \begin{lemma}{$($frozen cumulative processes$)$}\label{lmCumFzn}
  The processes $\hat{C}^n_f$ and $\hat{C}^n$ can be constructed on the same underlying space so that
  $d_{J_1} (\hat{C}^n_f, \hat{C}^n) \Ra 0$.
  \end{lemma}

  Now we want to establish the convergence $\hat{C}^n_f \Ra \hat{C}$ as $n \ra \infty$.
  To do so, we apply modified versions of the reasoning used to prove the
  FCLT in \eqn{cum}, as given in \cite{GW93}.
  In particular, as in (1.1)-(1.4) of \cite{GW93}, we observe that $\hat{C}^n_f$ is asymptotically equivalent to a random sum, ignoring remainder terms,
  and we then establish convergence for the sequence of random sums.
  To set the stage,
  let the $i^{th}$ full cycle in system $n$ end at time $T^n_i$.
  (Recall that the cycle begins upon transition into the designated state, while the next cycle begins upon first returning to that state after first leaving the state,
  which is well defined because the processes are pure-jump processes.)

  As in \S \ref{secFTSPinFCLT}, the key random variables associated with these cycles are the {\em cycle lengths}
  \bequ \label{cycleLengths}
  \tau^n_i \equiv T^n_i - T^n_{i-1}, \quad i \ge 1,
  \eeq
   and the integrals of the centered process over the cycle, which we call the {\em cycle variables},
  \bequ \label{cycleVar}
  Y^n_i \equiv \int_{T^n_{i-1}}^{T^n_i} \left(1_{\{D(\gamma_i, s) > 0\}} - \pi_{1,2} (\gamma_i) \right) \, ds, \quad i \ge 0,
  \eeq
  where $\gamma_i \equiv \eta(T^n_{i-1})$, with $T^n_{i-1}$ being the random time at which the $i^{\rm th}$ full cycle begins
  and $T^n_{0} =0$, so that $Y^n_0$ is the cycle variable for the first partial cycle.
  We do not need to make additional assumptions for the analog of the variables $W_i (f)$ in (1.2) of \cite{GW93} because
  \bequ \label{absCycle}
  W^n_i \equiv \int_{T^n_{i-1}}^{T^n_i} \left|1_{\{D(\gamma_i, s) > 0\}} - \pi_{1,2} (\gamma_i)\right| \, ds \le \tau^n_i.
  \eeq
  With this construction, we can write
  \bes
  \hat{C}^n_f (T^n_i; \eta) = \hat{C}^n (T^n_i; \tilde{\eta}^n_f), \quad i \ge 0,
  \ees
  for
  \bes
  \tilde{\eta}^n_f (t) = \gamma_i, \quad T^n_{i-1} \le t < T^n_i, \quad t < \ge 0.
  \ees

Unlike for a regenerative process, as in \cite{GW93}, here
  the random {\em cycle vectors} $(\tau^n_i, Y^n_i)$ are in general neither independent nor identically distributed.
  However, the sequence of cycle variables $\{(\tau^n_j, Y^n_j) : j \ge i\}$ is conditionally independent of the entire system
  history up to time $T^n_{i-1}$, which we denote by $\sF^n_{i-1}$, given only $T^n_{i-1}$, for each $i \ge 0$ and $n \ge 1$.
  Of course, in general these conditional distributions vary with $i$
  because the parameter state function $\eta$ is not constant, but they change little if $\eta$ changes little, by the QBD
  continuity.

  Let $N^n (t)$ count the number of full cycles up to time $t$.  As in (1.4) of \cite{GW93}, we can write
  \bes
  \hat{C}^n_f (t) = \hat{R}^n (t) + \hat{R}^n_1 (t) + \hat{R}^n_2 (t), \quad t \ge 0,
  \ees
  where
  $\hat{R}^n (t)$ is the random sum
  \bes
  \hat{R}^n (t) \equiv n^{-1/2} \sum_{i=1}^{N^n (t)} Y^n_i, \quad t \ge 0,
  \ees
  while $\hat{R}^n_1 (t)$ and $\hat{R}^n_2 (t)$ are remainder terms involving the initial and final partial cycle, if any,
  also scaled by dividing by $\sqrt{n}$.

   Just as in the standard regenerative setting, we are able to show that $\hat{C}^n_f$ is asymptotically equivalent to $\hat{R}^n$,
   so that it suffices to work with $\hat{R}^n$.
   \begin{lemma}{$($reduction to random sums$)$}\label{lmRandSum}
   As $n \ra \infty$, $\hat{R}^n_1 \Ra 0e$ and $\hat{R}^n_2 \Ra 0e$, so that
   $d_{J_1} (\hat{R}^n, \hat{C}^n_f) \Ra 0$.
   \end{lemma}

  It now suffices to show that $\hat{R}^n (\cdot; \eta) \Ra \hat{C} (\cdot; \eta)$ as $n \ra \infty$
  for each positive recurrent $\eta$ in $\D$.  By virtue of Corollary \ref{lmTvFCLTcont},
  given such an $\eta$, we can find a sequence of piecewise-constant state functions $\{\eta_{pc}^m: m \ge 1\}$
  where $\| \eta_{pc}^m - \eta \|_T \ra 0$ as $m \ra \infty$ with $\eta_{pc}^m$ being positive recurrent for all sufficiently large $m$.
 For those $m$, we have the desired convergence $\hat{C}^n (\cdot; \eta^m_{pc}) \Ra \hat{C} (\cdot; \eta^m_{pc})$ as $n \ra \infty$, as observed
 in the beginning of the proof.
  Thus, by Lemma \ref{lmCumFzn} and \ref{lmRandSum} above,
  we also have $\hat{R}^n (\cdot; \eta^m_{pc}) \Ra \hat{C} (\cdot; \eta^m_{pc})$ as $n \ra \infty$ for these $m$ as well.
  We now want to show that the established convergence also holds when $\eta^m_{pc}$ is replaced by $\eta$.
  For that purpose, we need to establish convergence as $n \ra \infty$ and $m \ra \infty$ jointly.  In order to justify that joint convergence,
  we establish the following result.

  \begin{lemma}{$($tightness and bounds for the random sums$)$}\label{lmRandSumBd}
  Consider a parameter state function $\eta$ in $\sD$ and a piecewise-constant parameter state function $\eta_{pc}$, where both
  $\eta$ and $\eta_{pc}$ are positive recurrent.  Let $T>0$ and $\delta > 0$ be such that $\|\eta - \eta_{pc}\|_T < \delta$.
  Then the sequence $\{\hat{R}^n (\cdot, \eta)\}$ is $C$-tight in $\D ([0,T^*])$ for some constant $T^* > 0$ and there exist
  functions $\sigma_l (\eta_{pc}(\cdot), \delta)$ and $\sigma_u (\eta_{pc}(\cdot), \delta)$
  such that the limit, say $\hat{R} (\cdot, \eta)$, of any convergent subsequence of $\{\hat{R}^n (\cdot, \eta)\}$ can be represented as
  \bequ \label{randSumRep}
  \hat{R} (t, \eta) = B (\bar{W} (t), \eta), \quad 0 \le t \le T^*,
  \eeq
  where $B$ is standard BM and $\bar{W}$ can be bounded above and below by
    \begin{eqnarray} \label{stBdsR}
  &&\int_{t_1}^{t_2} \sigma^2_{l} (\eta_{pc} (s), \delta) \, ds  \le  \bar{W} (t_2, \eta) - \bar{W} (t_1, \eta)
    \le  \int_{t_1}^{t_2} \sigma^2_{u} (\eta_{pc} (s), \delta) \, ds
   \end{eqnarray}
   for all $t_1$ and $t_2$ with $0 \le t_1 < t_2 \le T^*$,
   where $0 \le \sigma^2_{l} (\eta_{pc} (s), \delta) \le \sigma^2_{u} (\eta_{pc} (s), \delta) < \infty$ for all $s$, $0 \le s \le T$,
   and having the form in {\em \eqn{cumCycle}} determined by the state $\eta_{pc} (s)$.
Moreover, for any $\ep > 0$ and $T^* > 0$, there exist $\delta > 0$ and $T > 0$ as above, such that
   \bequ \label{epBd}
   \|\sigma^2_{u} (\eta_{pc} (\cdot), \delta) - \sigma^2_{l} (\eta_{pc} (\cdot), \delta)\|_{T^*} < \ep.
   \eeq
  \end{lemma}

  Lemma \ref{lmRandSumBd} is based on
 associated lemmas for partial sums from triangular arrays of the cycle lengths and cycle variables $\tau^n_i$ and $Y^n_i$,
 exploiting martingale structure; these results are stated in \S \ref{secMgFCLT} and proved in \S \ref{secProofsLemmas}.
Given  these lemmas, we now can complete the proof of Theorem \ref{lmTvFCLT}.  First, we have observed that
$\hat{C}^n (\cdot, \eta_{pc}) \Ra \hat{C} (\cdot, \eta_{pc})$ in $\D$ for any positive-recurrent piecewise-constant
parameter state function $\eta_{pc}$.
By Lemmas \ref{lmCumFzn}  and \ref{lmRandSum}, $\hat{R}^n (\cdot, \eta_{pc}) \Ra \hat{C} (\cdot, \eta_{pc})$ in $\D$
as well.
We can then apply Lemma \ref{lmRandSumBd} to deduce that the sequence of random sums $\{\hat{R}^n (\cdot, \eta)\}$ is tight.
 Hence, each subsequence has a convergent subsequence.  Let $\hat{R} (\cdot, \eta)$ be the limit of such a convergent subsequence.
    Next we construct a sequence $\{\eta^m_{pc}\}$ of positive-recurrent piecewise-constant state functions
  with $\|\eta^m_{pc} - \eta \|_T \ra 0$ as $m \ra \infty$.  As shown above,
  for each of them, we have $\{\hat{R}^n (\cdot, \eta^m_{pc})\} \Ra \hat{C} (\cdot, \eta^m_{pc})$
  as $n \ra \infty$.  However, again by Lemma \ref{lmRandSumBd}, we have $\hat{R} (\cdot, \eta)$ bounded above and below by the limits
  $\hat{C} (\cdot, \eta^m_{pc})$ which converge to $\{\hat{C} (\cdot, \eta)\}$ as $m \ra \infty$.
  Hence, we must have $\hat{R} (\cdot, \eta) = \{\hat{C} (\cdot, \eta)\}$.  Hence all convergent subsequences must have the same limit,
  which implies that we must have the full convergence, $\hat{R}^n (\cdot, \eta) \Ra \hat{C} (\cdot, \eta)$ in $\D$ as $n \ra \infty$.
  By Lemmas \ref{lmCumFzn}  and \ref{lmRandSum}, we must also have $\hat{C}^n (\cdot, \eta) \Ra \hat{C} (\cdot, \eta)$ in $\D$.
    Hence, Theorem \ref{lmTvFCLT} is proved.
    \qed

\section{Proof of Lemma \ref{lmRandSumBd}:  Using the Martingale FCLT}\label{secMgFCLT}

We have indicated that
Lemma \ref{lmRandSumBd} is based on
associated lemmas for partial sums from triangular arrays of the cycle lengths and cycle variables $\tau^n_i$ and $Y^n_i$,
exploiting martingale structure; in particular, we apply the martingale FCLT for triangular arrays.
We can treat these two components of $\hat{R}^n (\cdot, \eta)$ separately because,
just as in the familiar setting of renewal reward processes discussed in \S\S 7.4 and 13.2 of \cite{W02},
the FCLT for $\hat{R}^n (\cdot, \eta)$ depends on a FCLT for partial sums of $Y^n_i$ and a FWLLN for $N^n (t)$ separately.
By the inverse relation discussed in \S\S 7.3 and 13.6 of \cite{W02},
a FWLLN for $N^n (t)$ is equivalent to a corresponding FWLLN for the partial sums of $\tau^n_i$.
Since we can reduce the case of piecewise-constant $\eta_{pc}$ to the case of constant $\eta_c$ by focusing on the subintervals separately,
we now relate the given $\eta$ to a constant $\eta_c$.

Consider the cycle variables $Y^n_i$ in \eqn{cycleVar} associated with a parameter state function $\eta$.
Let $\sF^n_k$ be the $\sigma$-field generated by $X^n_6 (t): 0 \le t \le T^n_k$, $k \ge -1$.
Let
\bequ \label{Ymg1}
M^n_{Y} (k) \equiv \sum_{i=1}^{k} Y^n_i, k \ge 1, \qandq \hat{M}^n_{Y} (t) \equiv n^{-1/2} M^n_{Y} (\lfloor nt \rfloor), t \ge 0.
\eeq
For $i \ge 0$, let
\bequ \label{Ymg2}
\bsplit
\sigma^2_{n,i}  & \equiv  E[(Y^n_i)^2|\sF^n_{i-1}], \\
\bar{V}^n (t) & \equiv n^{-1} \sum_{i=1}^{\lfloor nt \rfloor} \sigma^2_{n,i}
\qandq  \mathcal{V}^n (t)  \equiv  \sup{\{s: \bar{V}^n (s) \le t\}}, \quad t \ge 0.
\end{split}
\eeq
We will be strongly exploiting the QBD continuity to obtain regularity in the variables $Y^n_i$.

\begin{lemma}{$($sums of cycle variables$)$}\label{lmCycleVarSumBd}
Consider a parameter state function $\eta$ in $\sD$ and an
associated constant parameter state function $\eta_c$, where
$\|\eta - \eta_{c}\|_T < \delta$ for some $T > 0$ and $\delta >0$,
and
both $\eta$ and $\eta_c$ are positive recurrent.
Consider the cycle variables $Y^n_i$ in {\em \eqn{cycleVar}}
 and the associated variables in {\em
\eqn{Ymg1}} and {\em \eqn{Ymg2}}, all associated with $\eta$.
Then there exist constants
$\sigma^2_l (\eta_c, \delta)$, $\sigma^2_u (\eta_c, \delta)$ and
$\delta' > 0$ such that, for all $i$ and $n$,
\bequ
\label{varBd} \sigma^2_l (\eta_c, \delta) \le \sigma^2_{n,i} \le
\sigma^2_u (\eta_c, \delta)
\qandq \sigma^2_u (\eta_c, \delta) - \sigma^2_l (\eta_c, \delta) < \delta',
\eeq
for $\sigma^2_{n,i}$
in {\em \eqn{Ymg2}}, associated with $\eta$, so that
\begin{eqnarray}\label{varBd2n}
\sigma^2_l (\eta_c, \delta) (t_2 - t_1)         & \le & \bar{V}^n (t_2) - \bar{V}^n (t_1) \le \sigma^2_u (\eta_c, \delta) (t_2 - t_1)
\end{eqnarray}
for all $n \ge 1$ and $0 \le t_1 < t_2 \le T$, for $\bar{V}^n$ in {\em \eqn{Ymg2}}.  As a consequence,
\begin{eqnarray}\label{varBd2nb}
\frac{(t_2 - t_1)}{\sigma^2_u (\eta_c, \delta)} & \le & \mathcal{V}^n (t_2) - \mathcal{V}^n (t_1) \le \frac{(t_2 - t_1)}{\sigma^2_l (\eta_c, \delta)}
\end{eqnarray}
for all $n \ge 1$ and $0 \le t_1 < t_2 \le T'$ for $T' \equiv T/\sigma^2_u (\eta, \delta)$.  Hence,
the sequences $\{\bar{V}^n\}$ and $\{\mathcal{V}^n\}$ associated with $\eta$, defined in {\em \eqn{Ymg2}}, are $C$-tight in $\D([0,T])$ and $\D([0,T'])$, respectively.
Moreover, the limits of convergent subsequences,
say $\bar{V}$ and $\mathcal{V}$ must satisfy corresponding inequalities, i.e.,
\begin{eqnarray}\label{varBd2}
\sigma^2_l (\eta_c, \delta) (t_2 - t_1)         & \le & \bar{V} (t_2) - \bar{V} (t_1) \le \sigma^2_u (\eta_c, \delta) (t_2 - t_1) \qandq \nonumber \\
\frac{(t_2 - t_1)}{\sigma^2_u (\eta_c, \delta)} & \le & \mathcal{V} (t_2) - \mathcal{V} (t_1) \le \frac{(t_2 - t_1)}{\sigma^2_l (\eta_c, \delta)}
\end{eqnarray}
for the same ranges of $t_1$ and $t_2$ above, so that $\bar{V}$ and $\mathcal{V}$ are both continuous and strictly increasing.
In addition,
\bequ \label{mgFCLT}
\hat{M}^n_{Y} \circ \mathcal{V}^n \Ra B \qinq \sD([0,T'])
\eeq
for $\hat{M}^n_{Y}$ in {\em \eqn{Ymg1}}, where $B$ is standard Brownian motion.
Thus, the sequence $\{\hat{M}^n_{Y}\}$ is $C$-tight in $\D([0,T])$ with the limit
of any convergent subsequence, say $\hat{M}_{Y}$, being of the form
\bequ \label{mgFCLT2}
\hat{M}_{Y} (t) = B (\bar{V} (t)), \quad 0 \le t \le T,
\eeq
where $\bar{V}$ is bounded above and below over all subintervals as in {\em \eqn{varBd2}}.
If we are free to choose the bounding constant $\delta$ above, then for any $\ep > 0$,
we can find $\delta > 0$ so that $\delta' < \ep$ for $\delta'$ in {\em \eqn{varBd}}.
\end{lemma}

We now state the corresponding result for the partial sums of the cycle lengths.

\begin{lemma}{$($sums of cycle lengths$)$}\label{lmCycleLengthSumBd}
 Consider a parameter state function $\eta$ in $\sD$ and a constant state function $\eta_c$, where both
  $\eta$ and $\eta_c$ are positive recurrent.  Let $T>0$ and $\delta > 0$ be such that $\|\eta - \eta_{c}\|_T < \delta$.
 Consider the cycle lengths $\tau^n_i$ in {\em \eqn{cycleLengths}} associated with $\eta$.
 Let $U^n_k \equiv \tau^n_1 + \cdots + \tau^n_k$, $k \ge 1$, and $\bar{U}^n (t) \equiv n^{-1} U^n_{\lfloor nt \rfloor}$, $t \ge 0$.
 Let $M^n_{U,i} \equiv E[\tau^n_i|\sF^n_{i-1}]$,
 $\bar{M}^n_{U} (t) \equiv n^{-1} (M^n_{U,1} + \cdots + M^n_{U,\lfloor nt \rfloor})$.
   Then the sequence $\{\bar{U}^n\}$ is $C$-tight in $\sD([0,T''])$ for an appropriate time $T'' > 0$, and if $\bar{U}$
 is the limit of a convergent subsequence, then necessarily it is bounded above and below with probability 1 by linear functions, i.e.,
 \bes
P( m_l (\eta_c, \delta) t \le \bar{U} (t) \le m_u (\eta_c,\delta) t, \quad 0 \le t \le T'') = 1.
 \ees
where $m_l (\eta_c, \delta)$ and $m_u (\eta_c, \delta)$ are constants depending on $\delta$ such that $0 < m_l (\eta_c, \delta) \le m_u (\eta_c, \delta) < \infty$.
If we are free to choose the time $T > 0$ and the bounding constant $\delta$ above, then for any $\ep > 0$ and $T''$, $0 < T'' < \infty$,
we can find $\delta > 0$ so that the conclusions above hold with $m_u (\eta_c, \delta) - m_l (\eta_c, \delta) < \ep$.
\end{lemma}

As a consequence of the inverse relation between the partial sums and the associated counting processes, as in Chapter 13 of \cite{W02},
we obtain the following corollary for the counting processes associated with the partial sums.
Let $\bar{N}^n (t) \equiv n^{-1} N^n (nt)$, $t \ge 0$.
In the next section we combine Corollary \ref{corCycleLengthCtBd} below with Lemma \ref{lmCycleVarSumBd} to prove Lemma \ref{lmRandSumBd}.
\begin{corollary}{$($counting process for cycle lengths$)$}\label{corCycleLengthCtBd}
Under the assumptions of Lemma {\em \ref{lmCycleLengthSumBd}},
the sequence of scaled counting processes $\{\bar{N}^n\}$ is $C$-tight in $\sD([0,T'''])$ for any time $T''' < T''/m_l (\delta)$, where $T''$ is as in
Lemma {\em \ref{lmCycleLengthSumBd}}.  If $\bar{N}$
 is the limit of a convergent subsequence of $\{\bar{N}^n\}$, then necessarily it is bounded above and below with probability 1 by linear functions, i.e.,
 \bequ \label{SObds3}
P( t/m_u (\eta_c,\delta) \le \bar{N} (t) \le  t/m_l (\eta_c,\delta), \quad 0 \le t \le T''') = 1.
 \eeq
where $m_l (\eta_c,\delta)$ and $m_u (\eta_c,\delta)$ are the constants depending on $\delta$ from Lemma {\em \ref{lmCycleLengthSumBd}} above.
If we are free to choose the time $T > 0$ and the bounding constant $\delta$ above, then for any $\ep > 0$ and $T'''$,
we can find $\delta > 0$ so that the conclusions above hold with $m_u (\eta_c,\delta) - m_l (\eta_c,\delta) < \ep$.
\end{corollary}

  \section{Remaining Proofs of Lemmas in \S\S \ref{secProofQBDFCLT} and \ref{secMgFCLT}}\label{secProofsLemmas}

In this section we prove five lemmas in the previous two sections, which were used in the
proof of Theorem \ref{lmTvFCLT}.  We prove them in the order needed for the proof.  We prove the one remaining lemma,
Lemma \ref{lmCumFzn} justifying the approximation by the frozen process $\hat{C}^n_f$,  afterwards in \S \ref{secFrozenPf}.

\paragraph{Proof of Lemma \ref{lmCycleVarSumBd}}
The key observation is that the sequence of random vectors $\{(\tau^n_j, Y^n_j): j \ge i\}$ associated with the general parametric state function $\eta$
is conditionally independent of the entire system
  history up to time $T^n_{i-1}$ for each $i$, which we have denoted by $\sF^n_{i-1}$, given only $T^n_{i-1}$.
  As a consequence, paralleling the regenerative case in \cite{GW93}
  and \eqn{cumCycle},
\bes 
E\left[\int_{T^n_{i-1}}^{T^n_i} \left(1_{\{D(\eta_i, s) > 0\}}\right) \, ds|\sF^n_{i-1}\right] =  \pi_{1,2} (\eta_i) E[\tau^n_i|\sF^n_{i-1}]
\ees
for $i \ge 1$, where $\eta_i \equiv \eta(T^n_{i-1})$, so that $E[Y^n_i|\sF^n_{i-1}] = 0$ for each $i$.
Hence, the stochastic process $\{M^n_{Y} (k): k \ge 1\}$ is a square integrable
martingale with respect to the filtration $\{\sF^n_k: k \ge 1\}$.

Moreover, by the QBD continuity, the variances $\sigma^2_{n,i} \equiv  E[(Y^n_i)^2|\sF^n_{i-1}]$ in \eqn{Ymg2}
cannot differ too much from the corresponding variance for the constant parameter function $\eta_c$.
For a fixed $t \ge 0$, let $\sigma^2_Y (\eta (t))$ be $\sigma^2_{n,i}$ under the condition that $T^n_{i-1} = t$,
so that $\eta_i \equiv \eta (T^n_{i-1}) = \eta(t)$.
Since $\|\eta - \eta_c\|_T < \delta$, we can apply the QBD continuity to
obtain the relations in \eqn{varBd}, where
\begin{eqnarray}\label{VarBds2}
\sigma^2_l (\eta_c, \delta) & \equiv & \min{\{\sigma^2_Y (\eta (t)): \eta  \in A(\eta_c, \delta)\}} \qandq \nonumber \\
\sigma^2_u (\eta_c, \delta) & \equiv & \max{\{\sigma^2_Y (\eta (t)): \eta  \in A(\eta_c, \delta)\}} \quad \mbox{with} \nonumber \\
A(\eta_c, \delta) & \equiv &\{\eta: \| \eta  - \eta_c \|_T \le \delta\}.
\end{eqnarray}
These in turn imply that the inequalities in \eqn{varBd2n} and \eqn{varBd2nb} hold for $\bar{V}^n$ and $\mathcal{V}^n$ for all $n$,
implying the tightness of the sequences $\{\bar{V}^n\}$ and $\{\mathcal{V}^n\}$ and the inequalities stated in \eqn{varBd2} for the limits of all convergent subsequences.
However, we cannot conclude that in general either $\bar{V}^n$ or $\mathcal{V}^n$ converges.

Nevertheless, we can apply an appropriate martingale FCLT to deduce that the limit in \eqn{mgFCLT} holds;
e.g., see Theorems 2.1 and 2.2 of \cite{DR78}, Theorem 5 of \cite{R77}
and p. 98 of \cite{HH80}.  The QBD continuity and the bounds
in \eqn{varBd} imply that the technical regularity conditions are satisfied in this case.
Hence, for any $\delta > 0$, we can apply the martingale FCLT to get the convergence in \eqn{mgFCLT}.

Given that $\mathcal{V}$ is a strictly increasing continuous function with bounded slope, as in \eqn{varBd2},
we can deduce from the tightness of $\{\hat{M}^n_{Y} \circ \mathcal{V}^n\}$, which follows from the convergence in \eqn{mgFCLT},
that the sequence $\{\hat{M}^n_{Y}\}$ itself must be tight.
That is most easily done by letting $\mathcal{V}^n$ be a continuous function constructed by linear interpolation
under which we still have the convergence in \eqn{mgFCLT}.  Then, $\bar{V}^n$ itself is a continuous strictly increasing function
with modulus bounds in \eqn{varBd2}.  Hence, we can deduce that the sequence $\{\hat{M}^n_{Y}\}$ must be tight.

The sequence $\{(\hat{M}^n_{Y}, \mathcal{V}^n, \bar{V}^n)\}$ is tight because the component sequences are all tight.
Starting from the joint convergence $(\hat{M}^n_{Y}, \mathcal{V}^n, \bar{V}^n) \Ra (\hat{M}_Y, \mathcal{V}, \bar{V})$ in $\sD_3$
for any convergent subsequence,
we can deduce from \eqn{mgFCLT} that $\hat{M}_Y = B \circ \bar{V}$, as claimed in \eqn{mgFCLT2}.
The final $\ep$ bound follows from the QBD continuity in Lemma C.5 of \cite{PeW10b}.
\hfill \qed 

\paragraph{Proof of Lemma \ref{lmCycleLengthSumBd}}
The proof is similar to the proof of Lemma \ref{lmCycleVarSumBd} above, but now we need a FWLLN instead of a FCLT.
However, it is convenient to apply the FCLT in order to deduce the FWLLN.
Indeed, by the same reasoning used to prove Lemma \ref{lmCycleVarSumBd} above, we can obtain a martingale FCLT
for the sums of the centered variables $\tau^n_i - E[\tau^n_i|\sF^n_{i-1}]$, paralleling \eqref{mgFCLT}.
Here we use the conditional variances and their sums, defined by
\begin{eqnarray*}
 && \sigma^2_{n,i}  \equiv  E[(\tau^n_i - E[\tau^n_i|\sF^n_{i-1}])^2|\sF^n_{i-1}], \quad \bar{V}^n (t) \equiv n^{-1} \sum_{i=1}^{\lfloor nt \rfloor} \sigma^2_{n,i}.
\end{eqnarray*}
instead of \eqn{Ymg2}.  We then obtain analogs of \eqn{varBd}, \eqn{varBd2n} and \eqn{varBd2}.

Given that FCLT,
 we scale further, essentially dividing by $\sqrt{n}$, to get the associated FWLLN for
the centered variables.  As a consequence, we obtain the FWLLN
$\bar{U}^n - \bar{M}^n_U \Ra 0e$ in $\D([0,T''])$ as $n \ra \infty$,
for an appropriate finite time $T''$, not necessarily equal to $T$ or $T'$ in the previous proof above.
Then, in direct analogy with \eqn{VarBds2}, we apply the
QBD continuity to obtain $m_l (\eta_c, \delta) \le M^n_{U,i} \le m_u (\eta_c, \delta)$
for all $i$ and $n$.  Hence, $m_l (\eta_c,\delta) t \le \bar{M}^n_U (t) \le m_u (\eta_c,\delta) t$
for all $n$ and $t$, $0 \le t \le T$.
We can then combine these bounds with the FWLLN to obtain the conclusions stated
in the lemma.
By the QBD continuity, $m_u (\eta_c,\delta) - m_l (\eta_c,\delta) \ra 0$ as $\delta \downarrow 0$.
\hfill \qed 

\paragraph{Proof of Lemma \ref{lmRandSumBd}}
First, Lemmas \ref{lmCycleVarSumBd} and \ref{lmCycleLengthSumBd} and Corollary \ref{corCycleLengthCtBd}
 can be extended directly to piecewise-constant state functions
as well as constant state functions.  Thus, for $\eta$ in $\D$, they imply that the sequences $\{\hat{M}^n_{Y}\}$ and $\{\bar{N}^{n}\}$ are each $C$-tight in $\D$.
 Consequently, the associated sequence of vector processes
$\{(\hat{M}^n_{Y},\bar{N}^{n})\}$ is $C$-tight in $\D_2$.  Hence, every subsequence has a further convergent subsequence.
Moreover, by Lemma \ref{lmCycleVarSumBd} and Corollary \ref{corCycleLengthCtBd}, any limit, say $(\hat{M}_{Y},\bar{N})$, can be represented as
$(B \circ \bar{V}, \bar{N})$, where $\bar{V}$ and $\bar{N}$ are bounded as in \eqn{varBd2} and \eqn{SObds3} over each subinterval
where the piecewise-constant parametric state function is constant.
Hence, overall they can be bounded above and below by
\bes
(\bar{V}_{Y,l},\bar{N}_l)  \le (\bar{V}_{Y},\bar{N}) \le  (\bar{V}_{Y,u},\bar{N}_u),
\ees
where $(\bar{V}_{Y,l}(0),\bar{N}_l(0)) = (\bar{V}_{Y,u}(0),\bar{N}_u(0)) = (0, 0)$ and
\begin{eqnarray}\label{vecBds2}
(\bar{V}_{Y,l} (t),\bar{N}_l (t))   & \equiv & (\bar{V}_{Y,l} (t_{i-1}) + \sigma^2_{l,i}  (t - t_{i-1}),\bar{N}_{l} (t_{i-1}) + (1/m_{u,i}) (t - t_{i-1})),  \nonumber \\
(\bar{V}_{Y,u} (t) ,\bar{N}_u (t))  & \equiv & (\bar{V}_{Y,u} (t_{i-1}) + \sigma^2_{u,i}  (t - t_{i-1}),\bar{N}_{u} (t_{i-1}) + (1/m_{l,i}) (t - t_{i-1})),  \nonumber
\end{eqnarray}
for $t_{i-1} \le t < t_i$, where $0 \equiv t_0 < t_1 < \ldots < t_k \equiv T$, so that $t_i$
are the endpoints of a piecewise constant state function $\eta_{pc}$, with $\sigma^2_{l,i}$ and $1/m_{u,i}$ being the lower bounds
and $\sigma^2_{u,i}$ and $1/m_{l,i}$ being the upper bounds on the $i^{\rm th}$
subinterval, depending on $\eta_{pc}$ and $\delta$.
Hence, we can apply the continuous mapping theorem to obtain the corresponding convergence
for the random sum for all convergent subsequences, with the limit of all convergent subsequences represented as claimed in \eqn{randSumRep}
with $\bar{W}$ there
 bounded as in \eqn{stBdsR}.  The bounding variance functions are given explicitly by
$\sigma^2_u (\eta_{pc} (s), \delta) = \sigma^2_{u,i}/m_{l,i}$ and
$\sigma^2_l (\eta_{pc}(s), \delta) = \sigma^2_{l,i}/m_{u,i}$ for $t_{i-1} \le s < t_i$.
Thus, by having $\| \eta - \eta_{pc}\|_T < \delta$ and choosing $\delta$ sufficiently small,
we can obtain the desired variance inequality \eqn{epBd}.
\hfill \qed 

\paragraph{Proof of Lemma \ref{lmRandSum}}
The reasoning follows the regenerative case as in \cite{GW93}.  First, the remainder term $\hat{R}^n_1 (t)$
is relatively easy to treat since it involves the initial cycle and is thus independent of $t$.
Since $D(\eta(0),0)$ has been specified as some fixed state after \eqn{cumTV1},
 the initial partial cycle until hitting time of the designated state
is clearly $O(1)$ and becomes asymptotically negligible when we divide by $\sqrt{n}$.

As in \cite{GW93}, to treat the second remainder term, we exploit the representation
\begin{eqnarray} \label{Rem2}
|\hat{R}^n_2 (t)| &\le & n^{-1/2} W^n_{N^n (t) + 1} \le n^{-1/2} \tau^n_{N^n (t) + 1} \nonumber \\
& \le & n^{-1/2} \max{\{ \tau^n_i: 1 \le i \le N^n(t)+1\}}, \quad t \ge 0,
\end{eqnarray}
for $W^n_i$ and $\tau^n_i$ defined in \eqn{absCycle} and \eqn{cycleLengths}.
However, the last term in \eqn{Rem2} is asymptotically negligible because of the FCLT for the cycle lengths
used in the proof of Lemma \ref{lmCycleLengthSumBd} above.  The last term is the maximum discontinuity in the prelimit process
indexed by $n$.  Since the limit is continuous, that term is asymptotically negligible.
\hfill \qed 

\section{Proof of Lemmas \ref{lmNewFrozen} and \ref{lmCumFzn}:  Coupling Constructions}\label{secFrozenPf}

In this section we prove the two lemmas justifying approximation by frozen processes, using coupling constructions.

\paragraph{Proof of Lemma \ref{lmNewFrozen}}
By the construction in \eqn{NewFrozen1},
we have forced the new frozen processes $\tilde{D}^n_{f}$ to coincide with the queue-difference processes $D^n_{1,2}$ for all
time points $t$ of the form $k/n$.
To complete the proof, we employ a special coupling construction to construct these two processes on the same underlying probability space
to make the processes have the same transitions within each interval $[(k-1)/n, k/n)$ with high probability.
As usual \cite{L92} \cite{W81}, this coupling construction produces an artificial joint distribution, but leaves the distributions of each of the two processes individually
unchanged.

We start by focusing on a single interval $[(k-1)/n, k/n)$. It
suffices to focus on one of these intervals, because we will show
that the construction is uniform over the $n$ intervals.
Since the transition rates in system $n$ are
of order $O(n)$ and the interval is of length $1/n$, it is
convenient to start by rescaling time as in the fluid limit in
Theorem \ref{th1}.
By doing a change of variables, we have
\begin{eqnarray*}
&& \sqrt{n} \int_{(k-1)/n}^{k/n} \left(1_{\{D^n_{1,2} (s) > 0\}}  - 1_{\{\tilde{D}^n_{f} (s) > 0\}}\right) \, ds,   \\
&& \quad \quad = \frac{1}{\sqrt{n}} \int_{0}^{1} \left(1_{\{D^n_{1,2} ((k-1)/n + s/n) > 0\}}  - 1_{\{\tilde{D}^n_{f} ((k-1)/n +s/n) > 0\}}\right) \, ds. \nonumber
\end{eqnarray*}
Then recall that both processes inside the integral converge appropriately to the FTSP.
To expose the connection, let $k$ go to infinity with $n$ so that $k/n \ra t$ as $n \ra \infty$.
First, by Theorem \ref{th1}, $\barx^n ((k-1)/n) \Ra x_6 (t)$.  Then, by Theorem 4.4 of \cite{PeW10b},
\bes
D^n_{1,2} ((k-1)/n + s/n) \equiv D^n_e(X^n((k-1)/n), s) \Rightarrow D(x_6 (t), s).
\ees
Second, by \eqn{frozenFTSP},
\begin{eqnarray*}
&& \{\tilde{D}^n_f((k-1)/n + s/n): 0 \le s \le 1\} \nonumber \\
&& \quad \deq  \{D(\lambda^n_i/n,m^n_j/n, X^n((k-1)/n), s): 0 \le s \le 1\} \nonumber \\
&& \quad  \Ra  \{D(x_6(t), s): 0 \le s \le 1\}.
\end{eqnarray*}
The main point for the coupling is that, after the change of time scale, both processes have transition rates
of order $O(1)$ that differ by $O(1/n)$.  Moreover, the processes are identical w.p.1 at the left end point of the interval $[0,1]$.

However, we need to apply the argument above to all $n$ intervals, where $n \ra \infty$.
It is thus important that the conclusions are valid uniformly over the $n$ subintervals.
Those conclusions are justified because the fluid limit in Theorem \ref{th1} implies that
$\bar{X}^n_6 \Ra x_6$ uniformly over each finite interval.  Moreover, the limit
$x_6$ is a continuous function over a bounded interval with values in in a compact subset of $\AA$.  Finally, the limiting transition rates are a continuous function of the state.

Let $\nu^n (T)$
be the number of $k$ for which the $nk \le T$ and the sample paths of
$\tilde{D}^n_{f}$ and $D^n_{1,2}$ fail to be identical over the interval $[(k-1)/n, k/n)$.
As a consequence of the asymptotically equivalent transition rates after changing the time scale above,
we show below that $\nu^n (T) = O(1)$ as $n \ra \infty$.
Thus, to complete the proof, we use the elementary bound
$\|\Delta^n \|_T \le \nu^n (T+ \ep)/\sqrt{n}$ for all $n \ge 1/\ep$, where $T>0$ and $\ep > 0$ are arbitrary constants.

We now discuss the coupling in more detail.
Since the transitions in the queue-difference process $D^n_{1,2}$ are generated from state changes in the CTMC $X^n_6$,
we do the special construction from the perspective of the CTMC $X^n_6$.
We use the device of uniformization to generate the transitions of the CTMC; i.e.,
we construct the transitions by thinning a Poisson process.
Without loss of generality, we use
different independent Poisson processes to generate potential transitions for each kind of transition, each interval $[(k-1)/n, k/n)$ and each $n$.
Since the transition rate of the CTMC is not uniformly bounded, there is a possibility that this direct construction
will be invalid, but by choosing these Poisson process rates sufficiently high, we can make the likelihood of a violation
asymptotically negligible.  In the actual construction, we can change the Poisson process when the constructed process
hits a state from which a further transition could lead to a violation.  The detailed construction does not matter
because we declare a difference occurring throughout the entire subinterval if the Poisson rate needs to be adjusted, thus contributing the maximum possible to the bound above.
Since the integrand in \eqn{NewFrozen12} is bounded by $1$, the total impact upon \eqn{NewFrozen12} by such rate violations
can clearly be made asymptotically negligible.

The coupling is achieved by using the same Poisson processes to generate the transitions in both $D^n_{1,2}$ and $\tilde{D}^n_{f}$
over each subinterval $[(k-1)/n, k/n)$.  These are done with respect to the states of $X^n_6 (t)$ and $X^n_6 ((k-1)/n)$.
For $D^n_{1,2}$, the transitions rates of the various transitions (arrivals, abandonments from each queue and service completions of each class from each pool)
are determined by the actual state $X^n_6 (t)$,
which changes throughout the interval $[(k-1)/n, k/n)$.  For, $\tilde{D}^n_{f}$, we do the same construction, but we leave the state fixed at its initial value
$X^n_6 ((k-1)/n)$ throughout the interval $[(k-1)/n, k/n)$, so that the transition rates do not change.  However, we match the transitions in the two systems
as much as possible.
We make the transitions differ only to the extent that the state of $X^n_6 (t)$ differs from $X^n_6 ((k-1)/n)$.

As stated above, we use different independent Poisson processes for each kind of transition.
We have one Poisson process generate potential arrivals for each $n$.  Since the arrival rates are unaffected by the state,
the Poisson process for generating potential arrivals of class $i$ can have rate $\lambda_i^n$, so that every potential arrival corresponds to an actual arrival in both systems.
Thus no difference is caused by any arrival.  That arrival in turn affects the constructed processes $D^n_{1,2}$ and $\tilde{D}^n_{f}$ in the obvious way:  an arrival of class
$1$ increases them by $1$, while an arrival of class $2$ decreases them by $r$.

For service completions of class $1$ by pool $2$, we let the Poisson process generating potential transitions have rate $\mu_{1,2} m^n_2$.
The actual transition rate at time $t$ for $X^n_6 (t)$ is $\mu_{1,2} Z^n_{1,2} (t)$, so that the Poisson rate is an upper bound
on the actual transition rate for all states.
If the Poisson process with rate $\mu_{1,2} m^n_2$ has a transition at time $t$, where $(k-1)/n \le t < k/n$, then
we let both systems have an actual service completion of class $1$ by pool $2$ at time $t$ with probability $(Z^n_{1,2} (t) \wedge Z^n_{1,2} ((k-1)/n)/ m^n_2$;
we let only the system associated with $D^n_{1,2}$ have an actual service completion of class $1$ by pool $2$ at time $t$ with probability $[Z^n_{1,2} (t) - (Z^n_{1,2} (t) \wedge Z^n_{1,2} ((k-1)/n)]/ m^n_2$;
we let only the system associated with $D^n_{f}$ have an actual service completion of class $1$ by pool $2$ at time $t$ with probability $[Z^n_{1,2} ((k-1)/n) - (Z^n_{1,2} (t) \wedge Z^n_{1,2} ((k-1)/n)]/ m^n_2$;
and we let neither system have an actual service completion of class $1$ by pool $2$ with probability $[m^n_2 - (Z^n_{1,2} (t) \vee Z^n_{1,2} ((k-1))/n)]/ m^n_2$.
Thus, a difference in the sample path is caused by this transition with probability $[(Z^n_{1,2} (t) \vee Z^n_{1,2} ((k-1)/n) - (Z^n_{1,2} (t) \wedge Z^n_{1,2} ((k-1)/n)]/m^n_2$,
which clearly is of order $O(1/n)$.

We do similar constructions with independent Poisson processes for each of the other transitions.
The abandonments are where the transition rate is unbounded, because the queue lengths $Q^n_i (t)$ are unbounded above.
However, the maximum queue length over the interval is bounded above by the initial queue length plus the number of arrivals over the interval,
so that the probability of violation is easily controlled by the Poisson arrival process for that class.
Hence, for the Poisson process generating potential abandonments from the class-$i$ queue over the interval $[(k-1)/n,k/n)$, we can give it
rate $(Q^n_i ((k-1)/n) + cn^3)\theta_i$ for $c > \lambda_i$.  (The exponent $3$ is chosen to make careful calculations unnecessary.)
This is sufficient, because the initial number in queue $i$ is $Q^n_i ((k-1)/n)$ and new class-$i$ arrivals occur at rate $\lambda^n_i$, which is $O(n)$.
The higher power of $n$ ensures that a violation of the rate-order uniformization condition is asymptotically negligible as $n \ra \infty$.
If the Poisson process generates a potential abandonment at time $t$, then it is a real abandonment for at least one system with probability $a_n/c_n = O(1/n^2)$,
a real abandonment for both systems
with probability $b_n/c_n = O(1/n^2)$ and a real abandonment for only one of the two systems with probability $(a_n - b_n)/c_n = O(1/n^3)$,
where $a_n \equiv Q^n_i ((k-1)/n) \vee Q^n_i (t)$, $b_n \equiv Q^n_i ((k-1)/n) \wedge Q^n_i (t)$ and
$c_n \equiv Q^n_i ((k-1)/n) + cn^3$.
The main point is that $(a_n - b_n) = O(1)$ because the two queues differ by arrivals at rate $O(n)$ over the interval of length $1/n$.
Hence, the probability that a {\bf real transition} at $t$ (not counting transitions from a state to itself, which are generated by
the common Poisson process) produces an abandonment for only one of the two systems is $(a_n - b_n)/a_n = O(1/n)$.
At the same time, the probability that the uniformization condition is violated during the entire interval is $o(1/n)$, so that it is asymptotically
negligible in the relevant scale.

We now assess the impact of this construction.
Both processes have transition rates of order $O(n)$ because the relevant processes $Q^n_i$ and $Z^n_{i,j}$ in $X^n_6$ are $O(n)$.
Thus, the processes $D^n_{1,2}$ and $\tilde{D}^n_{f}$ have $O(1)$ transitions over each interval of length $1/n$.
Hence, the state of $X^n (t)$ will only change an amount of order $O(1)$ within
each interval $[(k-1)/n, k/n)$.
Consequently, the probability of any one transition being different is $O(1/n)$, and the probability that there is any difference
over the interval $[(k-1)/n, k/n)$ is also of $O(1/n)$.
 Hence, $\nu^n (T)$ -- the total number of intervals having any difference over the interval $[0,T]$ -- will be of order $O(1)$, as claimed
 at the beginning of the proof.

Elaborating on the last step, observe that conditional upon $\bar{X}^n_6$, which converges to $x_6$,
we can regard $\nu^n (T)$ as the sum of at most $\lfloor nT \rfloor + 1$ independent Bernoulli random variable, assuming the value $1$
with probability $p_{n,i}$ and $0$ otherwise, where $p_{l}/n \le p_{n,i} \le p_{u}/n$ for all $i = 1, \dots, \lfloor nT \rfloor + 1$,
provided that $n$ is suitably large, where $p_{l}/n$ and $p_{u}/n$ are the minimum and maximum ``success probabilities'' among those Bernoulli random variables.
The bounds hold because $t \mapsto x_6(t)$ is a continuous function that is considered over a compact interval.
Hence, all the transition rates described above, producing the probabilities $p_{n,i}$ over each interval $i$, also have continuous limits which
can be bounded uniformly for all $n$ large enough.
Using the upper bound, we can bound $\nu^n (T)$ above stochastically by $\nu^n_{u} (T)$, defined as the partial sum of i.i.d. Bernoulli random variables.
taking the value $1$ with probability $p_{u}/n$.
By the LLN for partial sums from triangular arrays $\nu^n_{u} (T) \Ra p_{u}T$ as $n \ra \infty$, which implies that
$\nu^n (T)$ is indeed properly $O(1)$ as $n \ra \infty$.
Hence the proof is complete.
\hfill \qed 

\paragraph{Proof of Lemma \ref{lmCumFzn}}
The reasoning here is similar to the proof of Lemma \ref{lmNewFrozen}.
As before, we can use a coupling construction to make the two processes
have identical sample paths over the vast majority of the cycles.
We exploit the oscillation property for functions in $\D([0,T])$, Corollary 12.2 of \cite{W02},
concluding that, for any $\ep > 0$, there are $k$ time points $t_i$ with
$0 \equiv t_0 < t_1 < \cdots < t_{k-1} < t_k \equiv T$ such that $|\eta (s_1) - \eta (s_2)| < \ep$
for all $s_1, s_2 \in [t_{i-1}, t_i)$ for all $i$.
Hence, with the time scaling by $1/n$ in \eqn{cumTV1}, we see that,
except for at most $k$ cycles in $[0,T]$ containing the $k$ boundary points $t_i$,
the oscillation of $\eta$ over the cycle is at most $\ep/n$.
Hence, the coupling can be performed as in the proof of Lemma \ref{lmNewFrozen}, making
the probability that the sample paths differ over any one cycle among all except the $k$ be of order $O(1/n)$.
Since there are $O(n)$ cycles in $[0,T]$, as substantiated by Corollary \ref{corCycleLengthCtBd},
there are order $O(1)$ among the $O(n)$ cycles that have any difference in the sample paths.
Hence, with the spatial scaling by $\sqrt{n}$, we clearly have $d_{J_1} (\hat{C}^n_f, \hat{C}^n) \Ra 0$ as $n \ra \infty$
as claimed.
\hfill \qed 

\section{Comparisons with Simulation}\label{secSim}

To both support the validity of the theorems and their applicability to the intended
engineering problems, we now compare the approximations stemming from the FWLLN and the FCLT to
the results of simulation experiments.  Specifically, we will compare the Gaussian approximations for
the steady-state queue lengths with simulation estimates of these quantities, obtained by simulating the actual queueing model over a large time interval.
The approximate mean values come directly from the stationary point of the fluid limit, $x^*$ in Theorem \ref{thFluidStat};
the approximate variances come from Corollary \ref{corDiffLim}, specifically, from \eqn{covMatrix}.

Our simulation examples will have parameters related to a {\em base case}.
First, scale is described by the parameter $n$, which is the scaling parameter in our limit theorems.
The abandonment and service rate parameters, which describe the behavior of individual customers and servers, are
independent of $n$:  $\theta_1 = \theta_2 = 0.2$, $\mu_{1,1} = \mu_{2,2} = 1.0$ and
$\mu_{1,2} = \mu_{2,1} = 0.8$.  The service rates are chosen so that it is less efficient to serve a customer from
a different class.

The parameters that scale as the service system grows depend on $n$; they are chosen to
be directly proportional to $n$:
$m^{(n)}_i \equiv n m_i$, $\lambda^{(n)}_i \equiv n \lambda_i$ and $k^{(n)}_{1,2} \equiv n k_{1,2}$.
We take $k^n_{1,2}$ to be order $O(n)$ so it is easy to compare different system sizes.
(Note that the scaling of the thresholds is different than in Assumption 3.  This alternative choice facilitates
comparing the three different cases simulated.)
Our base case then has $m_1 \equiv m_2 \equiv 1$, $\lambda_1 = 1.3$, $\lambda_2 = 0.9$ and $k_{1,2} = 0.1$.
The arrival rates are chosen to put class $1$ in a focused overload, while class $2$ is initially
normally loaded or slightly underloaded, but becomes overloaded too after the sharing.  (These model parameters
satisfy case 1 of Assumption 3.1 of \cite{PeW10b}.)
We use the FQR-T control with ratio parameter $r = 1.0$, which allows us to apply the simple asymptotic formulas
from \S \ref{secR1}.

From \eqn{statPt} and \eqn{piSS}, we see that the stationary fluid solution for this base case yields
$z^*_{1,2} = 0.2111$, $q_1^* = 0.6556$, $q^*_2 = 0.5556$ and $\pi^*_{1,2} \equiv \pi_{1,2} (x^*) = 0.1763$.
Without any sharing, the fluid approximation for queue $1$ would be $1.5000$.  Hence the sharing reduces
the first fluid queue from $1.5000$ to $0.6556$, at the expense of causing the second class to have a fluid queue of $0.5556$.

We now turn to the variances, for which we need to analyze the FTSP more carefully.
The FTSP has BD parameters:  $\lambda_1 (x^*) = 1.411$, $\mu_1 (x^*) = 2.989$, $\lambda_2 (x^*) = 2.031$ and $\mu_2 (x^*) = 2.369$.
The associated $M/M/1$ traffic intensities are $\rho_1 (x^*) = 0.472$ and $\rho_2 (x^*) = 0.8574$.
The associated mean busy periods are $E[T_1 (x^*)] = 0.6338$ and $E[T_2 (x^*)] = 2.9603$.  Hence, the alternative formula for
$\pi_{1,2} (x^*)$ in \eqn{cum4} agrees with the value $0.1763$ given above (providing a check on our calculations).

Turning to the FCLT, from \eqn{psiDef}, we see that $\psi(x^*) = 0.6200$, so that $\psi^2 (x^*) = 0.3844$.
For $\sigma^2 (x^*)$, from \eqn{cum3}, we see that $E[T_1 (x^*)^2] = 1.5218$, so that $Var(T_1 (x^*)) = 1.1201$, and
$\sigma^2 (x^*) = 1.1201/3.5941 = 0.3116$.
Then $\xi_2  \equiv \psi^2(x^*)\sigma^2 (x^*) = 0.1198$.
Since $|\sM_{2,2}| = 0.176$, $\sZ_2 = 0.3403$.
Hence, $\sigma^2_{Z_{1,2}} (\infty) = 1 - 0.2111 + 0.3403 = 1.1292$.

As a consequence, $\sigma^2_{Q_s, Z_{1,2}} (\infty) = (1.1292)(0.5319) = 0.6006$.
Since $\mu_{2,2} - \mu_{1,2} = p_1 \theta_1 + p_2 \theta_2 = 0.2$, $\sQ_2 = \sigma^2_{Q_s, Z_{1,2}} (\infty) = 0.6006$.
Since $\sQ_1 = 11.0$, we have $\sigma^2_{Q_{s}} (\infty) = 11.6006$, so that the associated standard deviation is $3.41$.
(Without $\sQ_2$, we would approximate the standard deviation by $\sqrt{11} = 3.32$, so $\sQ_2$ contributes only $3\%$
to the standard deviation approximation in this case.)

By the SSC, the diffusion approximations for $Q_1$ and $Q_2$ are linearly related to $Q_s$; in particular,
$\sigma^2_{Q_{i}} (\infty) = (p_i)^2 \sigma^2_{Q_{s}} (\infty)$, so that $\sigma^2_{Q_{i}} (\infty) = 11.6006/4 = 2.900$
and the associated standard deviation is $1.70$.

We now turn to the simulations.  We simulate the actual queueing system obtained by scaling up the appropriate parameters by $n$.
We consider three cases:  $n = 25$, $n = 100$, and $n = 400$.
(Since $k_{1,2}^n$ must be an integer, we let $k_{1,2}^n = 3$ when $n = 25$.)

In all our simulation experiments, we used $5$ independent runs, each with $300,000$ arrivals.
We report averages together with the half widths of the $95\%$ confidence intervals,
based on a $t$ statistic with four degrees of freedom.
Simulation results for the base case above are presented in Table \ref{tbDist} below.

The first four rows of Table \ref{tbDist} show mean values.  We display
both the steady-state mean values and the associated scaled values (i.e.,
divided by $n$). The unscaled values helps us evaluate the performance of the actual system, while the scaled values
show the convergence in the FWLLN.  Table \ref{tbDist} clearly shows that the
accuracy improves as $n$ gets larger, but even for relatively small systems, the fluid approximation gives reasonable results.

\begin{table}[h!]
\begin{center}
\begin{tabular}{||ll|||c|c|||c|c|||c|c||}
\hline \hline
\multicolumn{2}{||c|||}{}   & \multicolumn{2}{c|||}{n=25}      & \multicolumn{2}{|c|||}{n=100}   & \multicolumn{2}{|c||}{n=400}\\ \hline \hline
     perf. meas.            &           & Approx.    & Sim.        & Approx.  & Sim.         & Approx.   & Sim.         \\ \hline
     $E[Q_1]$               &           & $16.6$     & $15.7$      & $65.6$   & $63.6$       & $262.2$   & $258.3$      \\
                            &           &            & $\pm 0.3$   &          & $\pm 1.9$    &           & $\pm 5.0$    \\ \hline
     $E[Q_1/n]$             &           & $0.656$    & $0.629$     & $0.656$  & $0.636$      & $0.656$   & $0.646$      \\
                            &           &            & $\pm 0.013$ &          & $\pm 0.019$  &           & $\pm 0.013$  \\ \hline
     $E[Q_2]$               &           & $13.6$     & $15.9$      & $55.6$   & $58.6$       & $222.2$   & $223.9$      \\
                            &           &            & $\pm 0.4$   &          & $\pm 1.8$    &           & $\pm 5.0$    \\ \hline
     $E[Q_2/n]$             &           & $0.556$    & $0.636$     & $0.556$  & $0.586$      & $0.556$   & $0.560$      \\
                            &           &            & $\pm 0.016$ &          & $\pm 0.018$  &           & $\pm 0.013$  \\ \hline \hline
     $std(Q_{s}) $          &           & $17.1$     & $16.0$      & $34.1$   & $33.7$       & $68.2$    & $67.6$       \\
                            &           &            & $\pm 0.3$   &         & $\pm 1.4$   &         & $\pm 2.9$    \\ \cline{3-8}
     $std(\hatq_{s})$       &           & $3.41$     & $3.21$      & $3.41$  & $3.37$      & $3.41$  & $3.38$       \\ \hline \hline
     $std(Q_1)$             &           & $8.5$      & $8.8$       & $17.0$  & $17.2$      & $34.0$  & $33.9$       \\
                            &           &            & $\pm 0.1$   &         & $\pm 0.7$   &         & $\pm 1.4$    \\ \cline{3-8}
     $std(\hatq_1)$         &           & $1.70$     & $1.75$      & $1.70$  & $1.72$      & $1.70$  & $1.70$        \\ \hline \hline
     $std(Q_2)$             &           & $8.5$      & $8.6$       & $17.0$  & $17.1$      & $34.0$  & $33.9$       \\
                            &           &            & $\pm 0.1$   &         & $\pm 0.7$   &         & $\pm 1.5$    \\ \cline{3-8}
     $std(\hatq_2)$         &           & $1.70$     & $1.73$      & $1.70$  & $1.71$      & $1.70$  & $1.69$       \\ \hline \hline
  \end{tabular}
\caption{A comparison of approximations to simulation results for the means and
standard deviations of the steady-state queue lengths as a function of the scale parameter $n$ in the
base case with $\lambda^n_1 = 1.3n$, $\lambda^n_2 = 0.9n$, $k^n_{1,2} = 0.1 n$, $r = 1$ and other parameters defined above.}
\label{tbDist}
\end{center}
\end{table}

Rows $5-10$ of Table \ref{tbDist} show
 the standard-deviations of the total queue length $Q_{s} = Q_1+Q_2$ as well as
the two queues.  As before, we treat both the actual values and the scaled values, but now
we are scaling in diffusion scale (dividing by $\sqrt{n}$ after subtracting the order-$O(n)$ mean), as in \eqn{DiffScale}, so that we will
be substantiating the FCLT, specifically Corollary \ref{corDiffLim} and the variance formulas in \eqn{covMatrix}.
 To save space,
we omit the confidence intervals for the scaled standard deviations; these can be computed from the confidence intervals of the actual
queues by dividing the half widths by $\sqrt{n}$.

Overall, we conclude that Table \ref{tbDist} shows that the approximations are remarkably accurate.

\section{Conclusions and Further Research} \label{secConclude}

In this paper we characterized the diffusion-limit refinements for the fluid limit of the X model operating under FQR-T.
Establishing the weak limits is nonstandard due to the effect of the stochastic AP, which contributes an additional Brownian term
that is independent of all other terms in the diffusion equation.

There are many open problems and directions for future research, to which we hope to contribute. One extension, mentioned at the end of \S \ref{secMain},
is to establish limits for corresponding non-Markovian $X$ models.  A second extension is to establish asymptotic optimality for
the FQR-T control within the optimization framework of \cite{PeW09a}, including a separable quadratic cost function,
under which FQR-T was shown to be optimal for the fluid model.
A third extension
is to design corresponding overload controls
for more complicated systems, possibly
 involving several customer classes and service pools.
 A fourth extension is
 to study the current system in a time-varying environment, in which arrival rates and/or staffing levels are
assumed to be time dependent.  Finally, it remains to seek new applications of the AP.
We anticipate that it will have many more applications in the future.

\section*{Acknowledgments}

This research is part of the first author's doctoral dissertation in the IEOR
Department at Columbia University. Additional work was done subsequently, including while the first
author had a postdoctoral fellowship at CWI in Amsterdam.
This research was partly supported by NSF grants DMI-0457095, CMMI 0948190 and CMMI 1066372.

\end{document}